\RequirePackage{amsthm}
\documentclass[sn-mathphys-num]{sn-jnl}

\usepackage{graphicx}%
\usepackage{multirow}%
\usepackage{amsmath,amssymb,amsfonts}%
\usepackage{amsthm}%
\usepackage{mathrsfs}%
\usepackage[title]{appendix}%
\usepackage{xcolor}%
\usepackage{textcomp}%
\usepackage{manyfoot}%
\usepackage{booktabs}%
\usepackage{algorithm}%
\usepackage{algorithmicx}%
\usepackage{algpseudocode}%
\usepackage{listings}
\usepackage{makecell}
\usepackage{subfigure}
\usepackage[section]{placeins}
\usepackage{float}
\usepackage{stfloats}
\theoremstyle{thmstyleone}%
\newtheorem{theorem}{Theorem}[section]
\newtheorem{proposition}[theorem]{Proposition}
\newtheorem{lemma}[theorem]{Lemma}%
\newtheorem{corollary}[theorem]{Corollary}

\theoremstyle{thmstyletwo}%
\newtheorem{remark}{Remark}[section]%

\theoremstyle{thmstylethree}%
\newtheorem{assumption}{Assumption}
\newtheorem{definition}{Definition}[section]%

\numberwithin{equation}{section}
\begin{document}

\title[A consensus-based optimization method for nonsmooth nonconvex programs]{A consensus-based optimization method for nonsmooth nonconvex programs with approximated gradient descent scheme}


\author[1]{\fnm{Jiazhen} \sur{Wei}}\email{jiazhenwei98@163.com}

\author[1]{\fnm{Fan} \sur{Wu}}\email{wufanmath@163.com}

\author*[1]{\fnm{Wei} \sur{Bian}}\email{bianweilvse520@163.com}

\affil[1]{\orgdiv{School of Mathematics}, \orgname{Harbin Institute of Technology}, \orgaddress{\city{Harbin}, \country{China}}}


\abstract{
In this paper, we are interested in finding the global minimizer of a nonsmooth nonconvex unconstrained optimization problem. By combining the discrete consensus-based optimization (CBO) algorithm and the gradient descent method, we develop a novel CBO algorithm with an extra gradient descent scheme evaluated by the forward-difference technique on the function values, where only the objective function values are used in the proposed algorithm. First, we prove that the proposed algorithm can exhibit global consensus in an exponential rate in two senses and possess a unique global consensus point. Second, we evaluate the error estimate between the objective function value on the global consensus point and its global minimum. In particular, as the parameter $\beta$ tends to $\infty$, the error converges to zero and the convergence rate is $\mathcal{O}\left(\frac{\log\beta}{\beta} \right)$. Third, under some suitable assumptions on the objective function, we provide the number of iterations required for the mean square error in expectation to reach the desired accuracy. It is worth underlining that the theoretical analysis in this paper does not use the mean-field limit.
Finally, we illustrate the improved efficiency and promising performance of our novel CBO method through some experiments on several nonconvex benchmark problems and the application to train deep neural networks.} 




\maketitle

\section{Introduction}\label{sec1}
Searching for the global minimizer $x^*$ of a potentially nonconvex objective function $f:\mathbb{R}^d\to\mathbb{R}$ is a long-standing problem and carries considerable significance across a wide range of applications in science, economy, engineering, as well as in the field of machine learning. However, computing global minimum $f^*$ or global minimizer $x^*$ is an NP-hard problem in general. More often than not, traditional optimization methods for solving such inherently challenging problems can only obtain local minimizers or even stationary points. Besides, numerous real-world problems face the challenges that the first-order derivative of the objective function may be either unavailable or expensive to calculate. For instance, obtaining the gradient of objective function is difficult in reinforcement learning \cite{Choromanski2018}, and training deep neural networks often encounters gradient explosion or vanishing \cite{Bengio1997}. Thus, gradient-based optimization algorithms are likely not suitable for global optimization. In recent years, as an alternative approximate optimization technique, a large number of meta-heuristic algorithms are proposed.

Meta-heuristic algorithms are a class of trustworthy and easy-to-understand optimization techniques employed to tackle highly nonconvex optimization problems. These encompass evolutionary algorithms \cite{Fogel2000}, genetic algorithms \cite{Tang1996}, particle swarm optimization \cite{Kennedy1995} and biology-inspired optimization algorithms \cite{Dorigo2005,Krishnanand2009,Karaboga2014}, which are rooted in the collective dynamics of biological swarm. By orchestrating an interaction between local and global strategies, as well as integrating stochastic and deterministic decisions, meta-heuristics generate a process capable of escaping the local minimizer and searching for the global minimizer. These algorithms are widely adopted in the fields such as image processing \cite{Karaboga2014}, machine learning \cite{Atyabi2013} and path planning \cite{Liang2020}. In spite of their significant empirical achievements and widespread implementation in various applications, the complexity of most meta-heuristic methods hinders them from establishing a robust mathematical framework that can demonstrate convergence under some reasonable assumptions.

Inspired by consensus dynamics and opinion formation, Pinnau et al. \cite{Pinnau2017} propose a new meta-heuristic algorithm, so-called consensus-based optimization (CBO), for solving nonconvex optimization problems. As a multi-agent system, CBO algorithm makes use of the interaction of finite particles to seek consensus around the global minimizer. Let $x^i(t)$ denote the position of the $i$-th particle at time $t$, for $i\in[N]$. Carrillo et al. \cite{Carrillo2021} study the following CBO algorithm: 
\begin{equation*}
	\begin{cases}
		\displaystyle \mathrm{d}x^i(t)=-\lambda\left (x^i(t)-\bar{x}^*(t)\right )\mathrm{d}t+\delta\sum\nolimits_{l=1}^d\left (x_l^i(t)-\bar{x}^*(t)\right )\mathrm{d}W^i_l(t)e_l,\\
		\displaystyle \bar{x}^*(t)=\dfrac{\sum_{i=1}^Nx^i(t)\mathrm{e}^{-\beta f(x^i(t))}}{\sum_{i=1}^N\mathrm{e}^{-\beta f(x^i(t))}},\\
	\end{cases}
\end{equation*}
where $\lambda>0$, $\delta\geq0$, $\beta>0$, $\{W^i_l(t)\}$ are independent standard Brownian motions for $i\in[N]$, $l\in[d]$, and $e_l\in\mathbb{R}^d$ is the $l$-th column of the identity matrix. 
In the above CBO algorithm, particles are pulled toward a consensus point by a deterministic drift term. This consensus point represents the instantaneous weighted average of all particle's current positions and serves as a temporary guess for the global minimizer. Additionally, due to Brownian motion, there is a stochastic diffusion term. This term encourages particles to explore the global minimizer, and the exploration capability is influenced by the distance between the particles and the consensus point. 

Lately, there are growing concerns on different variants of CBO algorithm and their rigorous theoretical analysis, e.g. Adam-CBO \cite{Chen2022}, CBO with memory effects \cite{Borghi2023,Riedl2024}, CBO with truncated noise \cite{Fornasier2024}, constrained CBO \cite{Fornasier2020,Fornasier2021a,Borghi2023a} and CBO for saddle point problems \cite{Huang2024}. The authors in the works mentioned above study the convergence of the algorithms, utilizing the Fokker-Planck equation, which can be derived in the mean-field model as the number of particles $N$ tends to infinity. However, since the Fokker-Planck equation is not the original model, these convergence results can not directly indicate the convergence of CBO algorithm itself \cite{Ha2020}. Therefore, showing a theoretical understanding of the finite particles system is highly significant.

In a recent work, Ha et al. \cite{Ha2021} propose a time-discrete CBO algorithm, namely vanilla-CBO, i.e.
\begin{equation*}
	\begin{cases}
		\displaystyle 
		x^{i,k+1}=x^{i,k}-\lambda\left( x^{i,k}-\bar{x}^{\star,k}\right) -\left( x^{i,k}-\bar{x}^{\star,k}\right)\odot \eta^k,\\
		\displaystyle 
		\bar{x}^{\star,k} =\frac{\sum_{i=1}^N x^{i,k}\mathrm{e}^{-\beta f(x^{i,k})}}{\sum_{i=1}^N \mathrm{e}^{-\beta f(x^{i,k})}},
	\end{cases}
\end{equation*}
where the operation $\odot$ denotes the Hadamard product, and  $\eta^k=\left(\eta_1^k,\ldots,\eta_d^k \right)^\top $ is a random vector with i.i.d. components and satisfying
\begin{equation*}
	\mathbb{E}[\eta_l^k]=0,\quad \mathbb{E}[|\eta_l^k|^2]=\delta^2,\quad \forall k\geq0,\quad l\in[d].
\end{equation*}
The authors in \cite{Ha2021} provide a framework to guarantee the convergence of the objective function values at $\{x^{i,k}\}_{k\geq0}$ toward the global minimum of $f$ without using the mean-field model. We also refer the readers to \cite{Ko2022,Ha2024} for the convergence analysis of other CBO variants with fixed number of particles $N$.

In this paper, we consider the following unconstrained optimization problem: 
\begin{equation}\label{problem}
	\min_{x\in\mathbb{R}^d} f(x),
\end{equation}
where $f:\mathbb{R}^d\rightarrow\mathbb{R}_+$ is continuous, but not necessarily convex or smooth, and the explicit expression of the gradient or subgradient for $f$ is not available, instead, only the function values of $f$ can be accessed. 
To further enhance the performance of the vanilla-CBO algorithm, we incorporate the gradient descent idea into our method. Specifically, we introduce an additional term to
describe the negative of the  gradient, which is a descent direction along the objective function values at the current particles. It is likely to accelerate the local convergence rate and push the algorithm move towards the minimizer. Moreover, in order to overcome the loss information of the gradients, we utilize the forward-difference method to approximate and describe the gradient information. Consequently, we develop a novel CBO algorithm that integrates the discrete CBO algorithm and derivative-free method. The rigorous convergence analysis of the proposed algorithm with fixed $N$ is given. More precisely, we first  prove that the iterate process of the algorithm ultimately converges to a global consensus point and then we discuss the optimality gap at this point. Moreover, we study the worst-case iteration complexity of this new algorithm.

\textbf{Contributions.} In view of the numerically observed better performance of the proposed algorithm compared to the vanilla-CBO algorithm in  \cite{Ha2020,Ha2021}, the proposed novel CBO algorithm in this paper is interesting and its convergence analysis is also of theoretical interest. The main contributions of this paper are as follows:
\begin{itemize}
	\item  We propose a new CBO algorithm to solve possibly nonconvex and nonsmooth optimization problem \eqref{problem}, based on the discrete CBO algorithm and derivative-free method. And unlike the works in \cite{Pinnau2017,Fornasier2021,Fornasier2022,Riedl2024}, we provide theoretical guarantees for the convergence of the proposed algorithm to the global minimizer $x^*$ without resorting to the corresponding mean-field limit.
	\item For Lipschitz continuous objective function $f$ in \eqref{problem}, we prove the convergence of the proposed algorithm toward the global minimum, whereas the result in \cite{Ha2020} requires that $f$ is a twice continuously differentiable function with bounded second-order derivative. Moreover, under some reasonable assumptions, the expected worst-case iteration complexity result is investigated. More precisely, the expected mean square distance between the iterates and global minimizer is less than or equal to $\varepsilon$ after at most $\mathcal{O}\left( \log(\varepsilon^{-1})\right) $ iterations. Note that the above complexity results are not provided in previous literatures \cite{Ha2020,Ha2021,Ko2022,Ha2024}.
	\item  We carry out numerical experiments to compare the performance of the proposed algorithm and the vanilla-CBO method in \cite{Ha2020,Ha2021} for minimizing some nonconvex benchmark functions. The results demonstrate that our method has a higher possibility of finding the global minimizer, outperforming the vanilla-CBO method. Furthermore, a fast algorithm is presented to solve the problems in deep learning, demonstrating its practicability.
\end{itemize}

Finally, we summarize the main results of this paper in Fig. \ref{fig9}.
\begin{figure}[h]
	\vspace*{-2mm}
	\centering
	\includegraphics[width=0.95\linewidth]{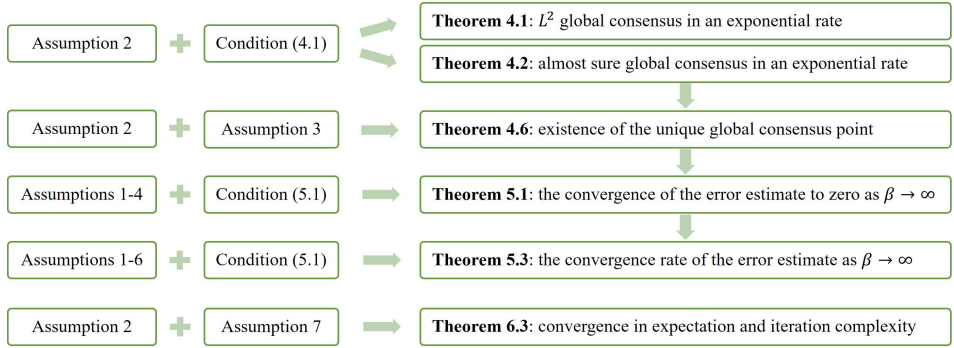}
	\caption{A brief description on the main results of this paper}
	\label{fig9}
\end{figure}

The remainder of this paper is organized as follows. In Section \ref{sec2}, we review some relevant preliminary results. In Section \ref{sec3}, we propose a novel CBO algorithm with approximated gradient descent scheme for solving problem \eqref{problem}. We study its corresponding stochastic global consensus results in Section \ref{sec4}, and provide the error estimate on the objective function values in Section \ref{sec5}. In Section \ref{sec6}, we study the worst-case iteration complexity result when $f$ satisfies some additional assumptions. In Section \ref{sec7}, numerical experiments are performed to demonstrate the efficiency of the proposed algorithm.
\section{Preliminaries}\label{sec2}
\subsection{Notations}\label{sec2.1}
 Let $\mathbb{R}^d$ denote the $d$-dimensional real-valued vector space and $\mathbb{R}^d_+=[0,+\infty)^d$. 
 $\left\| \cdot\right\|$ denotes the $\ell_2$ norm of the corresponding vector or matrix. 
We set $[N]:=\{1,\ldots,N\}$. $\odot$ represents the Hadamard product and $I_d$ denotes the $d\times d$ identity matrix.
For the random variables $X$ and $Y$ in the probability space $(\Omega,\mathcal{F},\mathbb{P})$, $X\sim Y$ means both $X$ and $Y$ follow the same distribution. $X\sim \mathcal{N}(m,s^2)$ denotes that $X$ follows a normal distribution with mean $m$ and variance $s^2$. $\mathbb{E}[X]$ and $\rm{Var}(X)$ indicate the expectation and variance of $X$, respectively. We write $\mathcal{P}_{ac}(\mathbb{R}^d)$ for the space of Borel probability measures that are absolutely continuous w.r.t the Lebesgue measure on $\mathbb{R}^d$. We use the abbreviations ``i.i.d." and ``a.s." for ``independent and identically distributed" and ``almost surely", respectively.  

\subsection{Probabilistic tools}\label{sec2.2}
At the beginning, we give the definitions of the stochastic consensus and submartingale for completeness. For a detailed study, the readers are referred to \cite{Klenke2008} and the references therein. 
\begin{definition}{\cite{Resnick1999}}
	Suppose $\{X_k\}_{k\geq0}$ and $X$ are random variables in $(\Omega,\mathcal{F},\mathbb{P})$. $\{X_k\}$ is said to converge to $X$ a.s., if there exists an event $A\in\mathcal{F}$ such that $\mathbb{P}(A)=0$ and $\lim_{k\to\infty}X_k(\omega)=X$ for all $\omega\notin A$.
\end{definition}
\begin{definition}\label{def-con}\cite{Ha2021}
	Let $\{X_k\}_{k\geq 0}=\left\{x^{i,k}: i\in[N]\right\}_{k\geq 0}$ be a discrete stochastic process.
	\begin{itemize}
		\item[{$(i)$}] The stochastic process $\{X_k\}_{k\geq 0}$ exhibits  $L^p$ global consensus with $p\geq 1$, if
		$$\lim_{k\rightarrow\infty}\max_{i,j\in[N]}\mathbb{E}\left[ \left\|x^{i,k}-x^{j,k} \right\|^p \right]=0.$$
		\item[{$(ii)$}] The stochastic process $\{X_k\}_{k\geq 0}$ exhibits almost sure global consensus, if
		$$\lim_{k\rightarrow\infty} \left\|x^{i,k}-x^{j,k} \right\|=0, ~~\forall  i,j\in[N],~a.s..$$
	\end{itemize}
\end{definition}

\begin{definition}{\cite{Resnick1999}}\label{def-mar}
	Given integrable random variables $\{X_k\}_{k\geq 0}$ and filtrations $\{\mathcal{F}_k\}_{k\geq 0}$, we call $\{(X_k,\mathcal{F}_k)\}_{k\geq 0}$ a submartingale if
	\begin{itemize}
		\item[{$(i)$}]  $X_k\in\mathcal{F}_k$ for each $k$, that is, $X_k$ is $\mathcal{F}_k$-measurable;
		\item[{$(ii)$}]  $\mathbb{E}[X_{k+1}|\mathcal{F}_k]\geq X_k$.
	\end{itemize}
\end{definition}

In what follows, we state some indispensable results which provide the basis for proofs of almost sure convergence.
\begin{theorem}[Doob's martingale convergence theorem \cite{Resnick1999}]\label{th-martingale}
	If $\{(X_k,\mathcal{F}_k)\}_{k\geq 0}$ is a submartingale and satisfies $\sup_k\mathbb{E}[|X_k|]<\infty$, then there exists $X_\infty$ such that $\mathbb{E}\left[ |X_\infty|\right]<\infty$ and $X_k\rightarrow X_\infty$ a.s..
\end{theorem}
\begin{proposition}[Borel-Cantelli lemma \cite{Resnick1999}]\label{pro-BC}
	Let $\{A_k\}_{k\geq0}$ be any events. If $\sum_k\mathbb{P}(A_k)<\infty$, then
\begin{equation*}
	\mathbb{P}\left( \limsup_{k\to\infty}A_k\right)=\mathbb{P}\left( \bigcap\limits_{k=1}^\infty\bigcup\limits_{n=k}^\infty A_n\right) =0. 
\end{equation*}
\end{proposition}
\begin{proposition}{\cite{Resnick1999}}\label{iff}
	$\{X_k\}$ converges to $X$ a.s. iff $\mathbb{P}\left( \bigcap\limits_{k=1}^\infty\bigcup\limits_{n=k}^\infty \left[ \left| X_n-X\right|\geq \varepsilon \right] \right) =0$ for all $\varepsilon>0$.
\end{proposition}
\begin{proposition}[Laplace's principle \cite{Pinnau2017}]\label{Laplace}
	Assume that $f:\mathbb{R}^d\rightarrow\mathbb{R}_+$ is continuous, bounded and attains the global minimum at the unique point $x^*$, then for any probability measure $\rho\in\mathcal{P}_{ac}(\mathbb{R}^d)$, it holds
	\begin{equation*}\label{laplace}
		\lim\limits_{\beta\rightarrow\infty}\left(-\frac{1}{\beta}\log\left(\int_{\mathbb{R}^d}\mathrm{e}^{-\beta f(x)}\mathrm{d}\rho(x)\right)\right)=f(x^*).
	\end{equation*}
\end{proposition}
\subsection{Gradient estimation}
 In problem \eqref{problem}, since the explicit expression of gradient for $f$ is not available, and only the function values are available, we use the finite difference methods to approximate its gradient. The most straightforward and economical finite difference method for approximating the first-order derivative is forward-difference approximation, which is defined by
\begin{equation}\label{fd}
	[g(x)]_l=\frac{f(x+\sigma e_l)-f(x)}{\sigma},\quad\forall l\in[d],
\end{equation}
where $\sigma>0$ is the finite difference interval and $e_l\in\mathbb{R}^d$ is the $l$-th column of the identity matrix.

\begin{lemma}\label{le-fd}
	If $f$ in \eqref{fd} is Lipschitz continuous with Lipschitz constant $L_f>0$, then there exist positive constants $M_g$ and $L_g$ such that
	\begin{align*}
		\left\|g(x)\right\|&\leq M_g, \\
		\left\|g(x)-g(y)\right\|&\leq L_g\|x-y\|, 
	\end{align*}
	for all $x,y\in\mathbb{R}^d$.
\end{lemma}
\begin{proof}
	By the Lipschitz continuity of $f$ and \eqref{fd}, it holds
	\begin{equation*}
		\left| \left[ g(x)\right]_l \right| \leq\frac{L_f}{\sigma}\left\|x+\sigma\mathrm{e}_l-x \right\|=L_f
	\end{equation*}
	and
	\begin{equation*}
		\left| \left[ g(x)\right]_l -\left[ g(y)\right]_l\right| \leq\frac{1}{\sigma}\left( \left|f(x+\sigma\mathrm{e}_l)-f(y+\sigma\mathrm{e}_l) \right| +\left| f(x)-f(y) \right|\right) \leq\frac{2L_f}{\sigma}\|x-y\|,
	\end{equation*}
	for all $l\in[d]$ and $x,y\in\mathbb{R}^d$.
	Then, we have
	\begin{equation*}
		\left\|g(x) \right\|=\sqrt{\sum_{l=1}^d\left| \left[ g(x)\right]_l \right|^2}\leq\sqrt{d}L_f=:M_g
	\end{equation*}
	and
	\begin{equation*}
		\left\|g(x)-g(y) \right\|=\sqrt{\sum_{l=1}^d\left| \left[ g(x)\right]_l-\left[ g(y)\right]_l \right|^2}\leq\frac{2\sqrt{d}L_f}{\sigma}\|x-y\|=:L_g\|x-y\|,
	\end{equation*}
	for all $x,y\in\mathbb{R}^d$.
\end{proof}	

\begin{remark}
	It is noteworthy that the convergence analysis of the proposed algorithm only requires that the gradient estimator $g$ to meet the properties given in Lemma \ref{le-fd}. While we focus on using forward-difference to estimate the gradient in this paper, alternative derivative-free techniques, such as central finite differences, polynomial interpolation, Gaussian smoothed gradient and ball smoothed gradient (refer to \cite{Berahas2022} for details), are also viable options for gradient estimation. Lemma \ref{le-fd} and Lemma \ref{le-fd} in the expected sense are applicable to these deterministic and stochastic gradient estimators, respectively. Thus, if \eqref{fd} is replaced with these estimators, the theoretical results of the proposed algorithm in this paper are still valid.
\end{remark}

\section{A CBO algorithm with an extra approximated gradient descent step}\label{sec3}

In this section, we propose a variant of CBO algorithm with an extra approximated gradient descent step. Specifically, for all particles, we first calculate new trial points via the CBO algorithm and then perform an additional approximated gradient step to obtain the next iterates. Here, we use the forward-difference, defined in \eqref{fd}, to approximate the gradient of the objective function at current iterate. A detailed description of our algorithm, namely Extra-Step Consensus-Based Optimization (ESCBO) algorithm, is given in Algorithm \ref{alg-escbo}. 

\begin{algorithm}[htp]
\caption{Extra-Step Consensus-Based Optimization (ESCBO) algorithm}
\label{alg-escbo}
\begin{algorithmic}[1]
	\Require
	the number of particles $N$, initial points $x^{i,0}\in\mathbb{R}^d$, $i\in[N]$,
	parameters $\beta>0$, $\lambda> 0$, $\delta\geq0$, $\sigma>0$ and a nonnegative sequence $\left\lbrace \alpha_k\right\rbrace $.\\
	Set $k=0$.
	\While{a termination criterion is not met,}
	\State Compute the weighted average point by
	\begin{equation}\label{weightedstate}
		\bar{x}^{\star,k}=\left( \bar{x}^{\star,k}_1,\ldots,\bar{x}^{\star,k}_d\right)^\top =\frac{\sum_{i=1}^N x^{i,k}\mathrm{e}^{-\beta f(x^{i,k})}}{\sum_{i=1}^N \mathrm{e}^{-\beta f(x^{i,k})}}.
	\end{equation}
	\State Generate a random vector $\eta^k=\left(\eta^k_1,\ldots,\eta^k_d \right)^\top $ whose components are i.i.d. and satisfies
	\begin{equation}\label{eta}
		\eta_l^k\sim\mathcal{N}(0,\delta^2),\quad l\in[d].
	\end{equation}
	\State Compute the gradient estimator $g^{i,k}:=g(x^{i,k})$ using \eqref{fd} for $i\in[N]$.
	\State Update $x^{i,k+1}$ for $i\in[N]$ by
    \begin{align}
			&y^{i,k+1}=x^{i,k}-\lambda\left( x^{i,k}-\bar{x}^{\star,k}\right) -\left( x^{i,k}-\bar{x}^{\star,k}\right) \odot\eta^k,\label{updatestate1} \\ 
			&x^{i,k+1}=y^{i,k+1}-\alpha_{k}g^{i,k}.\label{updatestate2}
	\end{align}
	\State Set $k=k+1$.
	\EndWhile
	\Ensure
	$x^{i,k+1}$, $i\in[N]$.
\end{algorithmic}
\end{algorithm}

The stochasticity in the ESCBO algorithm mainly results from the random variables $\{\eta_l^k\}_{k,l}$ in \eqref{eta}. Furthermore, we suppose that the probability space $(\Omega, \mathcal{F}, \mathbb{P})$ is sufficiently rich, enabling us to model the associated random processes in a unified way. We now define $\mathcal{F}_k$ as the $\sigma$-algebra generated by the family of random vectors $\eta^0,\ldots,\eta^k$, that is, 
\begin{equation}\label{Fk}
	\mathcal{F}_k:=\sigma\left(\eta^0,\ldots,\eta^k\right),\quad k\geq0.
\end{equation}
The design of the ESCBO algorithm yields that $x^{i,k}$ is adapted to the filtration $\mathcal{F}_{k-1}$, i.e. we have $x^{i,k}\in\mathcal{F}_{k-1}$ for all $i\in[N]$ and $k\geq0$.
Throughout this paper, we make use of the following assumptions.

\begin{assumption}\label{ass0}
	There exists a unique global minimizer $x^*\in\mathbb{R}^d$ to problem \eqref{problem}. 
\end{assumption}
\begin{remark}
	The uniqueness of the global minimizer in Assumption \ref{ass0} is a  common condition for analyzing the convergence of the CBO methods \cite{Ha2021,Ha2024,Fornasier2021,Fornasier2024,Carrillo2024}. This assumption stems from the definition of the weighted average point in \eqref{weightedstate}. Specifically, $\bar{x}^{\star,k}$ can be interpreted as an approximation of $\arg\min_{x\in\{x^{i,k},i\in[N]\}}f(x)$, whose approximation performance improves as $\beta\to\infty$, provided the global minimizer uniquely exists. If there are two global minimizers, $\bar{x}^{\star,k}$ could potentially lie between them, regardless of how large the value of $\beta$ is. Therefore, Assumption \ref{ass0} is essential for theoretical analysis.
\end{remark}
\begin{assumption}\label{ass1}
	$f$ in \eqref{problem} is Lipschitz continuous with Lipschitz constant $L_f>0$.
\end{assumption}
Under Assumption \ref{ass1} and Lemma \ref{le-fd}, let the random process $\{x^{i,k}\}_{i=1}^N$ be generated by the ESCBO algorithm, certain properties of gradient estimators $g^{i,k}$ are as follows:
\begin{align}
	\mathbb{E}\left[ \left\|g^{i,k}\right\|\right] &\leq M_g, \label{bound}\\
	\mathbb{E}\left[\left\|g^{i,k}-g^{j,k}\right\|\right] &\leq L_g\mathbb{E}\left[ \left \|x^{i,k}-x^{j,k}\right \|\right] , \label{lip}
\end{align}
for all $i,j\in[N]$ and $k\geq0$.

\section{Global consensus analysis}\label{sec4}
Global consensus means that all particles achieve an agreement on their common state by utilizing information derived from communication mechanisms. In this section, we are dedicated to analyzing the global consensus of the ESCBO algorithm under some appropriate conditions on parameters $\lambda$, $\delta$ and step sizes $\{\alpha_k\}_{k\geq0}$.
\subsection{Emergence of global consensus}\label{4.1}
According to Definition \ref{def-con}, we study several results related to the emergence of global consensus for the ESCBO algorithm in this subsection.
\begin{theorem}\label{th-L2}
	Let the random process $\{x^{i,k}\}_{i=1}^N$ be generated by the ESCBO algorithm. Suppose that Assumption \ref{ass1} is fulfilled and assume that the parameters $\lambda$, $\delta$ and step sizes $\{\alpha_k\}_{k\geq0}$ satisfy the following condition:
	\begin{equation}\label{condition1}
		(1-\lambda)^2+\delta^2<\frac{1}{2}\quad\mbox{and}\quad \lim_{k\to\infty}\alpha_k=0.
	\end{equation}
	Then, the ESCBO algorithm reaches $L^2$ global consensus in an exponential rate.
\end{theorem}
\begin{proof}
	Due to \eqref{updatestate1}, it follows
	\begin{equation*}
		\begin{split}
			y^{i,k+1}_l-y^{j,k+1}_l=&x^{i,k}_l-x^{j,k}_l-\lambda\left(x^{i,k}_l-x^{j,k}_l \right) -\left( x^{i,k}_l-x^{j,k}_l\right) \eta^k_l\\
			=&\left( 1-\lambda-\eta^k_l\right) \left( x^{i,k}_l-x^{j,k}_l\right).
		\end{split}
		\end{equation*}
	Then, this yields
	\begin{equation}\label{yij}
		\left| y^{i,k+1}_l-y^{j,k+1}_l\right|^2=\left( 1-\lambda-\eta^k_l\right)^2\left| x^{i,k}_l-x^{j,k}_l\right|^2.  
	\end{equation}
Taking expectation in \eqref{yij}, it holds that
\begin{equation}\label{Eyij}
	\begin{split}
			\mathbb{E}\left[ \left|y^{i,k+1}_l-y^{j,k+1}_l\right|^2\right] =&\mathbb{E}\left[\left( 1-\lambda-\eta^k_l\right)^2\left| x^{i,k}_l-x^{j,k}_l\right|^2\right] \\
			=&\mathbb{E}\left[\left( 1-\lambda-\eta^k_l\right)^2\right] \mathbb{E}\left[ \left| x^{i,k}_l-x^{j,k}_l\right|^2\right]\\
			=&\left((1-\lambda)^2+\delta^2 \right) \mathbb{E}\left[ \left| x^{i,k}_l-x^{j,k}_l\right|^2\right],
\end{split}		
	\end{equation}
where the second equality follows from the independence of $\eta_l^k$ and $x^{i,k}_l-x^{j,k}_l$ for any given $ k\geq 0$ and $l\in[d]$, and the last equality is due to \eqref{eta}.
Summing up \eqref{Eyij} over $l$, we have
\begin{equation}\label{Eyij2}
	\mathbb{E}\left[\left\|  y^{i,k+1}-y^{j,k+1}\right\| ^2\right]= \left((1-\lambda)^2+\delta^2 \right) \mathbb{E}\left[ \left\| x^{i,k}-x^{j,k}\right\|^2\right].
\end{equation}
Next, by \eqref{updatestate2}, we deduce
\begin{equation}\label{Exij2}
	\begin{split}
		\mathbb{E}\left[ \left\| x^{i,k}-x^{j,k}\right\|^2\right]=&\mathbb{E}\left[\left\| y^{i,k}-\alpha_{k-1}g^{i,k-1}-y^{j,k}+\alpha_{k-1}g^{j,k-1}\right\|^2  \right] \\
		\leq&2\mathbb{E}\left[\left\| y^{i,k}-y^{j,k}\right\|^2  \right]+2\alpha^2_{k-1}\mathbb{E}\left[\left\| g^{i,k-1}-g^{j,k-1}\right\|^2  \right]\\
		\leq&2\left((1-\lambda)^2+\delta^2 +\alpha_{k-1}^2L_g^2\right) \mathbb{E}\left[ \left\| x^{i,k-1}-x^{j,k-1}\right\|^2\right],
	\end{split}
\end{equation}
where the second inequality utilizes $(a+b)^2\leq 2(a^2+b^2)$, $\forall a,b\in\mathbb{R}$, and the last inequality follows from \eqref{lip} and \eqref{Eyij2}.
By the recursive relation in \eqref{Exij2}, we get
\begin{equation}\label{cn}
		\mathbb{E}\left[ \left\| x^{i,k}-x^{j,k}\right\|^2\right]\leq\prod_{n=0}^{k-1}C_{1,n}\mathbb{E}\left[ \left\| x^{i,0}-x^{j,0}\right\|^2\right],
\end{equation}
where $C_{1,n}:=2\left((1-\lambda)^2+\delta^2 +\alpha_{n}^2L_g^2\right)$.
Thus, combining this and \eqref{condition1}, it holds that if one chooses a positive constant $\tau_1$ satisfying 
\begin{equation*}
	2\left((1-\lambda)^2+\delta^2\right)<\tau_1<1,
	\end{equation*}
there exists a $K>0$ such that 
\begin{equation*}
	\mathbb{E}\left[ \left\| x^{i,k}-x^{j,k}\right\|^2\right]\leq\tau_1^k \mathbb{E}\left[ \left\| x^{i,0}-x^{j,0}\right\|^2\right],\quad \forall k\geq K.
\end{equation*}
In other words, there exists a positive and bounded random variable 
$$C_2(\omega):=\max_{0\leq k\leq K}\left\lbrace \tau_1^{-k}	\mathbb{E}\left[ \left\| x^{i,k}-x^{j,k}\right\|^2\right]\right\rbrace, $$
which satisfies
\begin{equation}\label{tau1}
	\mathbb{E}\left[ \left\| x^{i,k}-x^{j,k}\right\|^2\right]\leq C_2(\omega)\tau_1^k\leq C_2(\omega)\mathrm{e}^{(\tau_1-1)k},\quad \forall k\geq0,
\end{equation}
where the last inequality follows from $a\leq\mathrm{e}^{a-1},~\forall a\in\mathbb{R}$.
\end{proof}
In the subsequent corollary, we will establish a bound on the expectation of the difference between each particle's position and the weighted average point $\bar{x}^{\star}$, as the direct applications of consensus estimates provided by Theorem \ref{th-L2}. Before proceeding, we define the diameter at iteration $k$ as
$$\mathcal{D}\left( x^k\right) :=\max_{i,j\in[N],i\neq j}\left\|x^{i,k}-x^{j,k} \right\|^2. $$
\begin{corollary}
	Under the same setting as in Theorem \ref{th-L2} together with the following additional assumptions on the distribution of initial data:
	\begin{equation}\label{condition2}
		\{x^{i,0}\} ~\mbox{are i.i.d.},\quad x^{i,0}\sim x_{\rm{in}},
	\end{equation}
for all $i\in[N]$ and any given random vector $x_{\rm{in}}$, one has
\begin{equation}\label{var}
	\mathbb{E}\left[\left\|  x^{i,k}-\bar{x}^{\star,k}\right\| ^2\right] \leq2\prod_{n=0}^{k-1}C_{1,n}{\rm{Var}}(x_{\rm{in}}),
\end{equation}
where $C_{1,n}$ is defined as in \eqref{cn}.
\end{corollary}
\begin{proof}
	We apply the definition of $\bar{x}^{\star,k}$ in \eqref{weightedstate} to get
	\begin{equation}\label{Existar}
		\begin{split}
			\mathbb{E}\left[\left\|  x^{i,k}-\bar{x}^{\star,k}\right\| ^2\right]=&	\mathbb{E}\left[\left\|\frac{\sum_{j=1}^N\mathrm{e}^{-\beta f(x^{j,k})}\left(x^{i,k}-x^{j,k}\right)}{\sum_{j=1}^N\mathrm{e}^{-\beta f(x^{j,k})}} \right\|^2 \right]\\
				\leq&\mathbb{E}\left[\left(\frac{\sum_{j=1}^N\mathrm{e}^{-\beta f(x^{j,k})}\left\|x^{i,k}-x^{j,k}\right\|}{\sum_{j=1}^N\mathrm{e}^{-\beta f(x^{j,k})}} \right) ^2\right]\\
				\leq&\mathbb{E}\left[\frac{\sum_{j=1}^N\mathrm{e}^{-\beta f(x^{j,k})}\left\|x^{i,k}-x^{j,k}\right\|^2}{\sum_{j=1}^N\mathrm{e}^{-\beta f(x^{j,k})}}\right]\\
				\leq&\mathbb{E}\left[ \mathcal{D}\left(x^k\right) \right]\\
				\leq&\prod_{n=0}^{k-1}C_{1,n}\mathbb{E}\left[  \mathcal{D}\left( x^0\right) \right],
		\end{split}
	\end{equation}
where the first and the second inequalities follow from triangle inequality and  Cauchy-Schwarz inequality, respectively, and the last inequality is due to \eqref{cn}.
Moreover, by the condition in \eqref{condition2}, we have
\begin{equation}\label{Dx0}
	\begin{split}
		\mathbb{E}\left[  \mathcal{D}\left(x^0\right)\right]=&\mathbb{E}\left[ \max_{i,j\in[N],i\neq j}\left\|x^{i,0}-x^{j,0} \right\|^2 \right]\\
		=&\mathbb{E}\left[ \max_{i,j\in[N],i\neq j}\left( \left\|x^{i,0}\right\|^2+\left\|x^{j,0}\right\|^2-2\left(x^{i,0}\right)^\top x^{j,0}\right)  \right]\\
		=&2\mathbb{E}\left[ \|x_{\rm{in}}\|^2 \right] - 2\left( \mathbb{E}\left[x_{\rm{in}}\right] \right)^2 \\
		=&2{\rm{Var}}(x_{\rm{in}}).
	\end{split}
\end{equation}
We combine \eqref{Existar} and \eqref{Dx0} to obtain the final estimate \eqref{var}.
\end{proof}

In next theorem, we point out the almost sure global consensus of the ESCBO algorithm.
\begin{theorem}\label{th-as}
	Under the assumptions in Theorem \ref{th-L2}, the random process $\{x^{i,k}\}_{i=1}^N$ generated by the ESCBO algorithm admits almost sure global consensus in an exponential rate, i.e. there exists a constant $\tau_2>0$ such that 
	\begin{equation*}
		\lim_{k\to\infty}\mathrm{e}^{\tau_2k}\left\| x^{i,k}-x^{j,k}\right\|^2=0,\quad\mbox{a.s.},~ \forall i,j\in[N].
	\end{equation*}
	\end{theorem}
\begin{proof}
	By \eqref{tau1}, we can choose a positive constant $\tau_2$ satisfying $\tau_1+\tau_2<1$, then for all $i,j\in[N]$, it holds
	\begin{equation*}
		\mathbb{E}\left[ \mathrm{e}^{\tau_2k}\left\| x^{i,k}-x^{j,k}\right\|^2\right]\leq C_2(\omega)\mathrm{e}^{(\tau_1+\tau_2-1)k},\quad \forall k\geq0.
	\end{equation*}
This yields
\begin{equation*}
	\sum_{k=0}^\infty\mathbb{E}\left[ \mathrm{e}^{\tau_2k}\left\| x^{i,k}-x^{j,k}\right\|^2\right]\leq\sum_{k=0}^\infty C_2(\omega)\mathrm{e}^{(\tau_1+\tau_2-1)k}<\infty.
\end{equation*}
Combining the above inequality with Markov inequality, we have
\begin{equation*}
	\sum_{k=0}^\infty\mathbb{P}\left( \mathrm{e}^{\tau_2k}\left\| x^{i,k}-x^{j,k}\right\|^2\geq\varepsilon \right)\leq\frac{\sum_{k=0}^\infty\mathbb{E}\left[ \mathrm{e}^{\tau_2k}\left\| x^{i,k}-x^{j,k}\right\|^2\right]}{\varepsilon} <\infty,\quad\forall \varepsilon>0.
\end{equation*}
Then, applying the Borel-Cantelli lemma in Proposition \ref{pro-BC}, it holds
\begin{equation}\label{bc}
	\mathbb{P}\left( \bigcap\limits_{k=1}^\infty\bigcup\limits_{n=k}^\infty\left[\mathrm{e}^{\tau_2n}\left\| x^{i,n}-x^{j,n}\right\|^2\geq\varepsilon \right] \right) =0,\quad\forall \varepsilon>0.
\end{equation}
By Proposition \ref{iff}, \eqref{bc} is equivalent to
\begin{equation*}
	\lim_{k\to\infty}\mathrm{e}^{\tau_2k}\left\| x^{i,k}-x^{j,k}\right\|^2=0,\quad\mbox{a.s.}.
\end{equation*}
 Consequently, we have the desired result.
\end{proof}

\subsection{Emergence of the consensus point}\label{sec4.2}
 According to Theorems \ref{th-L2} and \ref{th-as}, $\left\|x^{i,k}-x^{j,k} \right\|$  is shown to decay exponentially in expectation and in almost sure sense, which can not guarantee that the solution process converges to a common point, since it is possible for the solution to approach a limit cycle or display chaotic behavior. In this subsection, we will study the existence of the global consensus point of the ESCBO algorithm.

Let $\{x^{i,k}\}_{i=1}^N$ be a solution process to the ESCBO algorithm. For the following theoretical analysis, we set
\begin{equation}\label{AkBk}
	A^{i,k}:=\sum_{n=0}^k\left\|x^{i,n}-\bar{x}^{\star,n} \right\| \quad \mbox{and} \quad B^{i,k}:= \sum_{n=0}^k\left\|\left( x^{i,n}-\bar{x}^{\star,n}\right)\odot\eta^n  \right\|,
\end{equation}
for all $i\in[N]$ and $k\geq0$.
\begin{lemma}\label{le-mar}
	Consider $\mathcal{F}_k$ defined in \eqref{Fk}, then for all $i\in[N]$, $\left\{\left( A^{i,k},\mathcal{F}_{k-1}\right)\right\}_{k\geq0}$ and $\left\{\left( B^{i,k},\mathcal{F}_{k}\right)\right\}_{k\geq0}$ are submartingales.
\end{lemma}
\begin{proof}
	Following \eqref{updatestate1} and \eqref{updatestate2}, it is clear that
	\begin{equation}\label{adapt}
		A^{i,k}\in\mathcal{F}_{k-1} \quad\mbox{and}\quad B^{i,k}\in\mathcal{F}_{k}.
	\end{equation}
In addition, for all $i\in[N]$ and $k\geq0$, it yields
\begin{equation*}
		\mathbb{E}\left[ A^{i,k+1}\big|\mathcal{F}_{k-1}\right]=\mathbb{E}\left[ A^{i,k}\big|\mathcal{F}_{k-1}\right]+\mathbb{E}\left[\left\| x^{i,k+1}-\bar{x}^{\star,k+1}\right\|\Big|\mathcal{F}_{k-1}\right]\geq A^{i,k},
\end{equation*}
where the last inequality follows from \eqref{adapt}.
Again, the proof for random sequence $B^{i,k}$, which is analogous to $A^{i,k}$, is provided as follows.
\begin{equation*}
		\mathbb{E}\left[ B^{i,k+1}\big|\mathcal{F}_{k}\right]
		=\mathbb{E}\left[ B^{i,k}\big|\mathcal{F}_{k}\right]+\mathbb{E}\left[\left\|\left(  x^{i,k+1}-\bar{x}^{\star,k+1}\right)\odot\eta^{k+1}\right\| \Big|\mathcal{F}_{k}\right]
		\geq B^{i,k}.
\end{equation*}
Therefore, we know from Definition \ref{def-mar} that $\left\{\left( A^{i,k},\mathcal{F}_{k-1}\right)\right\}_{k\geq0}$ and $\left\{\left( B^{i,k},\mathcal{F}_k\right)\right\}_{k\geq0}$ are submartingales.
\end{proof}
\begin{lemma}\label{le-bound}
	Under the assumptions in Theorem \ref{th-L2}, the expectations of random sequences $A^{i,k}$ and $B^{i,k}$ are uniformly bounded for all $k\geq0$, that is, 
	\begin{equation*}
		\sup_{k}\mathbb{E}\left[ A^{i,k}\right]<\infty\quad \mbox{and}\quad  \sup_{k}\mathbb{E}\left[ B^{i,k}\right]<\infty,\quad\forall i\in[N].
	\end{equation*}
\end{lemma}
\begin{proof}
	Using the definition of $A^{i,k}$ in \eqref{AkBk}, we have
	\begin{equation}\label{EA}
		\mathbb{E}\left[ A^{i,k}\right]=\mathbb{E}\left[\sum_{n=0}^k\left\|x^{i,n}-\bar{x}^{\star,n} \right\| \right]\leq\sum_{n=0}^k\sqrt{\mathbb{E}\left[\left\|x^{i,n}-\bar{x}^{\star,n} \right\|^2 \right] }\leq  \sum_{n=0}^k\sqrt{\mathbb{E}\left[\mathcal{D}\left(x^n \right) \right] },		
	\end{equation}
where the first inequality is due to Jensen's inequality for $\left(\mathbb{E}[\cdot]\right)^2$, and the last inequality follows from \eqref{Existar}. As in the proof of Theorem \ref{th-L2}, the expectation of diameter process $\mathbb{E}\left[ \mathcal{D}(x^n)\right] $ decays to zero exponentially. Hence, the relations in \eqref{EA} imply
\begin{equation*}
	\sup_{k}\mathbb{E}\left[ A^{i,k}\right]\leq\sum_{n=0}^\infty\sqrt{\mathbb{E}\left[\mathcal{D}\left(x^n \right) \right]}<\infty,\quad\forall i\in[N].
\end{equation*}
In addition, we calculate the square of each term in $B^{i,k}$ as
\begin{equation*}
	\left\|\left( x^{i,n}-\bar{x}^{\star,n}\right)\odot\eta^n  \right\|^2=\sum_{l=1}^d\left( \left( x^{i,n}_l-\bar{x}^{\star,n}_l\right)\eta^n_l \right) ^2.
\end{equation*}
Taking expectation on both side of the above equation and following \eqref{eta} and \eqref{Existar}, we obtain
\begin{equation}\label{Exeta}
	\begin{split}
		\mathbb{E}\left[ \left\|\left( x^{i,n}-\bar{x}^{\star,n}\right)\odot\eta^n  \right\|^2\right]=& \sum_{l=1}^d\mathbb{E}\left[\left( \left( x^{i,n}_l-\bar{x}^{\star,n}_l\right)\eta^n_l \right) ^2\right] 
		=\delta^2\sum_{l=1}^d\mathbb{E}\left[ \left( x^{i,n}_l-\bar{x}^{\star,n}_l\right)^2\right]\\
		=&\delta^2\mathbb{E}\left[\left\|x^{i,n}-\bar{x}^{\star,n} \right\|^2  \right]
		\leq\delta^2\mathbb{E}\left[ \mathcal{D}\left(x^n\right)\right] .	
	\end{split}
\end{equation}
By Jensen's inequality and \eqref{Exeta}, one has
\begin{equation}\label{EB}
	\begin{split}
		\mathbb{E}\left[ B^{i,k}\right] =&\sum_{n=0}^k\mathbb{E}\left[ \left\|\left( x^{i,n}-\bar{x}^{\star,n}\right)\odot\eta^n  \right\|\right]\\
		\leq&\sum_{n=0}^k\sqrt{\mathbb{E}\left[ \left\|\left( x^{i,n}-\bar{x}^{\star,n}\right)\odot\eta^n  \right\|^2\right]}\leq\delta\sum_{n=0}^k\sqrt{\mathbb{E}\left[ \mathcal{D}\left(x^n\right)\right]}.
	\end{split}
\end{equation}
Thus, similar to the proof of $\mathbb{E}\left[A^{i,k} \right] $, by \eqref{EB}, we have
\begin{equation*}
	\sup_{k}\mathbb{E}\left[ B^{i,k}\right]\leq\delta\sum_{n=0}^\infty\sqrt{\mathbb{E}\left[\mathcal{D}\left(x^n \right) \right]}<\infty,\quad\forall i\in[N].
\end{equation*}
\end{proof}

Following the above two lemmas, we obtain our main result on the uniqueness of the global consensus point to the ESCBO algorithm.

\begin{assumption}\label{ass2}
	The parameters $\lambda$, $\delta$ and step sizes $\{\alpha_k\}_{k\geq0}$ in the ESCBO algorithm satisfy
	\begin{equation*}
		(1-\lambda)^2+\delta^2<\frac{1}{2}\quad\mbox{and}\quad \sum_{k=0}^\infty\alpha_k<\infty.
	\end{equation*}
\end{assumption}

\begin{theorem}\label{th-xinfty}
	Let the random process $\{x^{i,k}\}_{i=1}^N$ be generated by the ESCBO algorithm and suppose that Assumptions \ref{ass1} and \ref{ass2} are satisfied. 
Then, there exists a random vector $x_\infty$ such that
\begin{equation*}
	\lim_{k\to\infty}x^{i,k}=x_\infty,\quad{\mbox{a.s.}},~\forall i\in[N].
\end{equation*}
\end{theorem}
\begin{proof}
	To begin with, we denote $G^{i,k}:=\sum_{n=0}^k\alpha_n\left\| g^{i,n}\right\|$.
	Note that \eqref{updatestate1} and \eqref{updatestate2} can be equivalently rewritten as
	\begin{equation}\label{xk+1}
		x^{i,k+1}-x^{i,k}=-\lambda\left( x^{i,k}-\bar{x}^{\star,k}\right)- \left( x^{i,k}-\bar{x}^{\star,k}\right)\odot\eta^k-\alpha_kg^{i,k},\quad\forall k\geq0.
	\end{equation}
Then, for all $m,n\geq0$, we have 
\begin{equation}\label{xk0}
	\begin{split}
		\left\| x^{i,m}-x^{i,n}\right\| &\leq\sum_{k=n}^{m-1}\left\| x^{i,k+1}-x^{i,k}\right\|\\
		&\leq\sum_{k=n}^{m-1}\lambda \left\| x^{i,k}-\bar{x}^{\star,k} \right\| +\sum_{k=n}^{m-1}\left\| \left( x^{i,k}-\bar{x}^{\star,k}\right)\odot\eta^k\right\|+\sum_{k=n}^{m-1}\alpha_k\left\| g^{i,k}\right\| \\
		&=\lambda \left(A^{i,m-1}-A^{i,n-1} \right) +\left(B^{i,m-1}-B^{i,n-1} \right)+\left(G^{i,m-1}-G^{i,n-1} \right).
	\end{split}
\end{equation}
Similar to the proof in Lemma \ref{le-mar}, it is clear that $\left\{\left( G^{i,k},\mathcal{F}_{k-1}\right)\right\}_{k\geq0}$ is a submartingale. 
Moreover, following from \eqref{bound} and Assumption \ref{ass2}, we have
\begin{equation*}
	\sup_{k}\mathbb{E}\left[ G^{i,k}\right]\leq\sum_{n=0}^{\infty}\alpha_n\mathbb{E}\left[ \left\| g^{i,n}\right\|\right] \leq M_g\sum_{n=0}^{\infty}\alpha_n<\infty,
\end{equation*}
which together with Lemmas \ref{le-mar} and \ref{le-bound} yields that $\left\{\left( A^{i,k},\mathcal{F}_{k-1}\right)\right\}_{k\geq0}$, $\left\{\left( B^{i,k},\mathcal{F}_k\right)\right\}_{k\geq0}$ and $\left\{\left( G^{i,k},\mathcal{F}_{k-1}\right)\right\}_{k\geq0}$ are uniformly bounded submartingales. Therefore, by Doob's martingale convergence theorem in Theorem \ref{th-martingale}, there exist random variables $A^i_\infty$, $B^i_\infty$ and $G^i_\infty$ such that
\begin{equation*}
	\lim_{k\to\infty}A^{i,k}=A^i_\infty, \quad \lim_{k\to\infty}B^{i,k}=B^i_\infty,\quad \lim_{k\to\infty}G^{i,k}=G^i_\infty,\quad{\mbox{a.s.}},~\forall i\in[N].
\end{equation*}
Thus, $\left\lbrace A^{i,k}\right\rbrace_{k\geq0}$, $\left\lbrace B^{i,k}\right\rbrace_{k\geq0}$ and $\left\lbrace G^{i,k}\right\rbrace_{k\geq0}$ are Cauchy for a.s. $\omega\in\Omega$. For any given $\varepsilon>0$, there exists a sufficiently large $K$ such that for all $m,n>K$, 
\begin{equation}\label{cauchy}
	\max\left\lbrace\left|A^{i,m}-A^{i,n} \right|, \left|B^{i,m}-B^{i,n} \right|, \left|G^{i,m}-G^{i,n} \right| \right\rbrace\leq\varepsilon, \quad{\mbox{a.s.}},~\forall i\in[N].  
\end{equation}
Therefore, \eqref{xk0} and \eqref{cauchy} imply that $\{x^{i,k}\}_{k\geq0}$ is Cauchy. Then, there exists a $x^i_\infty$ such that $\lim_{k\to\infty}x^{i,k}=x^i_\infty$, a.s., for all $i\in[N]$. This  along with the global consensus result in Theorem \ref{th-as} deduces that there exists a random vector $x_\infty$ such that $x^{i,k}$ converges to $x_\infty$ a.s. for all $i\in[N]$.
\end{proof}

\begin{corollary}
	Let the assumptions in Theorem \ref{th-xinfty} hold. Then,
	\begin{itemize}
		\item[{$(i)$}] there exists a positive constant $M_f$ such that for all $i\in[N]$ and $k\geq0$, $$\left| f(x^{i,k})\right| \leq M_f,\quad \mbox{a.s.};$$
		\item[{$(ii)$}] $\lim_{k\to\infty}\bar{x}^{\star,k}=x_\infty$, a.s., where $x^\infty$ is the same as in Theorem \ref{th-xinfty}.
	\end{itemize}
\end{corollary}
\begin{proof}
	$(i)$ From Theorem \ref{th-xinfty}, it implies that the random process $\{x^{i,k}\}_{i=1}^N$ is bounded with probability $1$. This along with the continuity of $f$ leads to the boundedness of $f(x^{i,k})$.
	
	$(ii)$ This assertion is a direct consequence of \eqref{weightedstate} and Theorem \ref{th-xinfty}.
\end{proof}

\section{Almost sure convergence}\label{sec5}
In this section, we provide the analysis of the optimal gap at the global consensus point, namely $\left| {\rm ess}\inf f(x_\infty)-f(x^*)\right| $. We prove that given suitable assumptions concerning the parameters and the initial distribution of particles in the ESCBO algorithm, this error estimate is bounded by a quantity based on the parameter $\beta$.

\subsection{Convergence analysis of the ESCBO algorithm}\label{sec5.1}
We first consider the case that $f$ is Lipschitz continuous, that is, Assumption \ref{ass1} holds. Hereinafter, we make use of the following additional assumption.

\begin{assumption}\label{ass3}
Let $x_{\rm{in}}$ be a reference random variable with a law $\rho$ that satisfies $\rho\in\mathcal{P}_{ac}(\mathbb{R}^d)$. The initial data of the ESCBO algorithm satisfies that $\{x^{i,0}\}$ are i.i.d. with $x^{i,0}\sim x_{\rm{in}}$ for each $i\in[N]$.
\end{assumption}
To ease notation, let $f^*:=f(x^*)$. And we now derive the error estimate between the almost sure limit of objective function values along the ESCBO algorithm and the global minimum of problem \eqref{problem}.

\begin{theorem}\label{th-err}
	 Assume that Assumptions \ref{ass0}-\ref{ass3} are satisfied. For some $\epsilon\in(0,1)$, if the parameters $\{\beta, \lambda, \delta\}$, step sizes $\{\alpha_k\}_{k\geq0}$ and the initial distribution in the ESCBO algorithm satisfy
	\begin{equation}\label{condition4}
		(1-\epsilon)\mathbb{E}\left[ \mathrm{e}^{-\beta f(x_{\rm{in}})}\right] \geq \beta L_fC_3\mathrm{e}^{-\beta f^*},
	\end{equation}
 where $C_3=\sum_{n=0}^{\infty}\left((\lambda+\delta)\sqrt{2\prod_{m=0}^{n-1}C_{1,m}{\rm{Var}}(x_{\rm{in}})}+\alpha_nM_g \right)$, and $C_{1,m}$ is defined in \eqref{cn}. Then, it holds
\begin{equation*}
	\left| {\rm ess}\inf_{\omega\in\Omega}f(x_\infty(\omega))-f^*\right| \leq E(\beta)\quad\mbox{and}\quad \lim_{\beta\to\infty}E(\beta)=0,
\end{equation*}
where $x_\infty$ is the global consensus point defined in Theorem \ref{th-xinfty}.
\end{theorem}
\begin{proof}
	To begin with, for all $k\geq0$, we deduce
	\begin{equation}\label{exk+1}
		\begin{split}
			\frac{1}{N}\sum_{i=1}^N\mathrm{e}^{-\beta f(x^{i,k+1})}-\frac{1}{N}\sum_{i=1}^N\mathrm{e}^{-\beta f(x^{i,k})}
			=&\frac{1}{N}\sum_{i=1}^N\mathrm{e}^{-\beta f(x^{i,k})}\left( \mathrm{e}^{-\beta \left( f(x^{i,k+1})-f(x^{i,k})\right) }-1\right) \\
			\geq&\frac{1}{N}\sum_{i=1}^N\mathrm{e}^{-\beta f(x^{i,k})}(-\beta)\left( f(x^{i,k+1})-f(x^{i,k})\right)\\
			\geq&-\frac{\beta L_f\mathrm{e}^{-\beta f^*}}{N}\sum_{i=1}^N\left\|x^{i,k+1}-x^{i,k} \right\|, 
		\end{split}
	\end{equation}
where the first inequality uses the relation $a\leq\mathrm{e}^a-1,~\forall a\in\mathbb{R}$, and the last inequality follows from Assumption \ref{ass1}.
We take expectation on both sides of \eqref{xk+1} to get
\begin{equation}\label{Exk+1}
	\begin{split}
		\mathbb{E}\left[ \left\| x^{i,k+1}-x^{i,k}\right\| \right]\leq&\lambda\mathbb{E}\left[\left\| x^{i,k}-\bar{x}^{\star,k}\right\|\right] +\mathbb{E}\left[\left\|  \left( x^{i,k}-\bar{x}^{\star,k}\right)\odot\eta^k\right\| \right] +\alpha_k\mathbb{E}\left[\left\| g^{i,k}\right\| \right] \\		
		\leq&(\lambda+\delta)\sqrt{\mathbb{E}\left[\left\|x^{i,k}-\bar{x}^{\star,k} \right\|^2 \right] }+\alpha_kM_g\\
		\leq&(\lambda+\delta)\sqrt{2\prod_{n=0}^{k-1}C_{1,n}{\rm{Var}}(x_{\rm{in}})}+\alpha_kM_g,
	\end{split}
\end{equation}
where the second inequality is due to Jensen's inequality, \eqref{bound} and \eqref{Exeta}, and the last inequality follows from \eqref{var}. Next, summing up \eqref{exk+1} over $k$, and applying expectation on the resulting relation, we obtain
\begin{equation}\label{Eek0}
	\begin{split}
		&\mathbb{E}\left[\frac{1}{N}\sum_{i=1}^N\mathrm{e}^{-\beta f(x^{i,k})} \right]\\
		\geq&  \mathbb{E}\left[\frac{1}{N}\sum_{i=1}^N\mathrm{e}^{-\beta f(x^{i,0})} \right]-\frac{\beta L_f\mathrm{e}^{-\beta f^*}}{N}\sum_{i=1}^N\sum_{n=0}^{k-1}\mathbb{E}\left[\left\|x^{i,n+1}-x^{i,n} \right\|\right] \\
		\geq&\mathbb{E}\left[\mathrm{e}^{-\beta f(x_{\rm{in}})} \right] -\beta L_f\mathrm{e}^{-\beta f^*}\sum_{n=0}^{k-1}\left((\lambda+\delta)\sqrt{2\prod_{m=0}^{n-1}C_{1,m}{\rm{Var}}(x_{\rm{in}})}+\alpha_nM_g \right) ,
	\end{split}
\end{equation}
where the last inequality uses \eqref{Exk+1} and Assumption \ref{ass3}.
Now, by the conditions in this theorem and the definition of $C_{1,m}$ in \eqref{cn}, we know
\begin{equation*}
	\sum_{n=0}^{\infty}\sqrt{\prod_{m=0}^{n-1}C_{1,m}}<\infty \quad \mbox{and}\quad \sum_{n=0}^{\infty}\alpha_nM_g<\infty.
\end{equation*}
Thus, it is clear that $\sum_{n=0}^{\infty}\left((\lambda+\delta)\sqrt{2\prod_{m=0}^{n-1}C_{1,m}{\rm{Var}}(x_{\rm{in}})}+\alpha_nM_g \right)<\infty$, and we set its limit by $C_3$.
Then letting $k\rightarrow\infty$ in \eqref{Eek0}, by Theorem \ref{th-xinfty} and \eqref{condition4}, we get
\begin{equation*}
		\mathbb{E}\left[\mathrm{e}^{-\beta f(x_\infty)} \right]\geq\mathbb{E}\left[\mathrm{e}^{-\beta f(x_{\rm{in}})} \right] -\beta L_fC_3\mathrm{e}^{-\beta f^*}\geq\epsilon\mathbb{E}\left[ \mathrm{e}^{-\beta f(x_{\rm{in}})}\right].
\end{equation*}
Therefore, we have
\begin{equation*}\label{log}
	\mathrm{e}^{-\beta{\rm ess}\inf\limits_{\omega\in\Omega}f(x_\infty(\omega))}=\mathbb{E}\left[ \mathrm{e}^{-\beta{\rm ess}\inf\limits_{\omega\in\Omega}f(x_\infty(\omega))}\right] \geq\mathbb{E}\left[ \mathrm{e}^{-\beta f(x_\infty)}\right] \geq\epsilon\mathbb{E}\left[ \mathrm{e}^{-\beta f(x_{\rm{in}})}\right] .
\end{equation*}
We take the logarithm and divide by $-\beta$ to both side of the previous inequality to get
\begin{equation}\label{essinf}
		{\rm ess}\inf_{\omega\in\Omega}f(x_\infty(\omega))\leq
		-\frac{1}{\beta}\log\mathbb{E}\left[ \mathrm{e}^{-\beta f(x_{\rm{in}})}\right] -\frac{1}{\beta}\log\epsilon.
\end{equation}
So by Proposition \ref{Laplace}, if we define
\begin{equation*}\label{Ebetaa}
	E(\beta):=-\frac{1}{\beta}\log\mathbb{E}\left[ \mathrm{e}^{-\beta f(x_{\rm{in}})}\right] -f^*-\frac{1}{\beta}\log\epsilon,
\end{equation*}
then
\begin{equation*}
	\left| {\rm ess}\inf_{\omega\in\Omega}f(x_\infty(\omega))-f^*\right| \leq E(\beta),
\end{equation*}
where $\lim_{\beta\rightarrow\infty}E(\beta)=0$.
\end{proof}

\subsection{Convergence rate under stronger condition}\label{sec5.2}
In subsection \ref{sec5.1}, we show that the error function $E(\beta)$ converges to zero. In this subsection, we focus on the convergence rate of error estimate to zero as $\beta\to\infty$ under the case that $f$ in \eqref{problem} is twice continuously differentiable. 
Frist, we list the standard setting for the error analysis in this part.
\begin{assumption}\label{ass4}
	$f$ in \eqref{problem} is twice continuously differentiable and locally strictly convex at $x^*$, that is, $\nabla^2f(x^*)$ is positive definite, where $x^*$ is the unique global minimizer of \eqref{problem}. 
\end{assumption}
\begin{assumption}\label{ass5}
	The probability density function $\phi$ of random vector $x_{\rm{in}}$  defined as in Assumption \ref{ass3}, has a compact support set, is continuous at $x^*$ and satisfies $\phi(x^*)>0$.
\end{assumption}
\begin{lemma}{\cite{Ha2021}}\label{le-rate}
	Under Assumptions \ref{ass0} and \ref{ass3}-\ref{ass5}, we have
	\begin{equation*}
		-\frac{1}{\beta}\log\mathbb{E}\left[ \mathrm{e}^{-\beta f(x_{\rm{in}})}\right]=f^*+\frac{d\log\beta}{2\beta}+\mathcal{O}\left(\frac{1}{\beta}\right), \quad \beta\to\infty, 
	\end{equation*}
where $x_{\rm{in}}$ is given in Assumption \ref{ass3}.
\end{lemma}
In what follows, based on Theorem \ref{th-err} and Lemma \ref{le-rate}, we provide the following error estimate.
\begin{theorem}\label{th-rate}
	Under the same assumptions as in Theorem \ref{th-err} together with Assumptions \ref{ass4} and \ref{ass5}, given the point $x_\infty$ defined in Theorem \ref{th-xinfty}, one has
	\begin{equation*}
		\left| {\rm ess}\inf_{\omega\in\Omega}f(x_\infty(\omega))-f^*\right| \leq \frac{d\log\beta}{2\beta}+\mathcal{O}\left(\frac{1}{\beta}\right), \quad \beta\to\infty.
	\end{equation*}
\end{theorem}
\begin{proof}
	By \eqref{essinf} and Lemma \ref{le-rate}, we have
	\begin{equation*}
		\begin{split}
			{\rm ess}\inf_{\omega\in\Omega}f(x_\infty(\omega))\leq&
			f^*+\frac{d\log\beta}{2\beta}+\mathcal{O}\left(\frac{1}{\beta}\right) -\frac{1}{\beta}\log\epsilon\\
			=&
			f^*+\frac{d\log\beta}{2\beta}+\mathcal{O}\left(\frac{1}{\beta}\right), \quad \beta\to\infty.
		\end{split}
	\end{equation*}
	which gives the conclusion.
\end{proof}
From Theorem \ref{th-rate}, we know that the convergence rate of the error $\left| {\rm ess}\inf_{\omega\in\Omega}f(x_\infty(\omega))-f^*\right|$ is $\mathcal{O}\left(\frac{\log\beta}{\beta} \right)$.

\section{Iterate convergence in expectation}\label{sec6}
In Section \ref{sec5}, we provide the error estimation between the objective function value at the consensus point and the global minimum in the almost sure sense. Here, we focus on the convergence of iterates generated by the ESCBO algorithm in expectation under some additional assumptions on $f$ in \eqref{problem}.
Inspired by \cite{Fornasier2022}, we define the mean square distance between iterate $x^{i,k}$ of the ESCBO algorithm and the unique minimizer $x^*$, given by
\begin{equation}\label{Wk}
	W_k=\frac{1}{N}\sum_{i=1}^N\left\|x^{i,k}-x^* \right\|^2 ,
\end{equation}
and provide an upper bound to the number of iterations that the algorithm takes until $\mathbb{E}\left[ W_k\right] \leq\varepsilon $, where $\varepsilon$ is a given positive number.

\begin{assumption}\label{ass6}
	For the unique minimizer $x^*$ of \eqref{problem}, there exist positive constants $f_\infty$, $R_0$, $\nu$ and $\mu$ such that
	\begin{align}
		\mu\left\|x-x^*\right\|&\leq\left(f(x)-f^*\right)^\nu, \quad \forall x\in\mathcal{B}_{R_0}\left(x^*\right), \label{inverse}\\
		f_\infty&<f(x)-f^*,\quad\forall x\notin\mathcal{B}_{R_0}\left(x^*\right),\label{finfty}
	\end{align}
where $\mathcal{B}_{R_0}\left(x^*\right):=\left\{x\in\mathbb{R}^d:\left\|x-x^*\right\|\leq R_0\right\}$ denotes the ball of radius $R_0$ centered at $x^*$.
\end{assumption}
	There exist some functions satisfying Assumption \ref{ass6}. For example, we consider the function $f:\mathbb{R}\to\mathbb{R}$ defined by
	\begin{equation*}
		f(x)= x^2-10\cos(2\pi x)+10.
	\end{equation*}
	Such objective function is nonconvex and satisfies Assumption \ref{ass6} with $f_\infty=1$, $R_0=1$, $\nu=\frac{1}{2}$ and $\mu=1$.
\begin{remark}
	Assumption \ref{ass6} is often used for the analysis of CBO algorithms \cite{Fornasier2021,Fornasier2022,Fornasier2024,Carrillo2024,Riedl2024}. Condition \eqref{inverse} guarantees that $f$ is locally coercive in a neighborhood of $x^*$ and gives a lower bound on the local growth of $f$. Such property is proposed in \cite{Fornasier2021}, namely inverse continuity condition, and also known as a quadratic growth condition with $\nu=\frac{1}{2}$ in \cite{Necoara2019}. Condition \eqref{finfty} avoids the case that $f(x)$ is close to $f^*$ for some $x$ far from $x^*$. It is noteworthy that Assumption \ref{ass6} yields the uniqueness of the global minimizer $x^*$ requested in Assumption \ref{ass0}.
\end{remark}

 In the following, $\mathcal{B}_{R_0}\left(x^*\right)$ is further abbreviated to $\mathcal{B}_{R_0}$, and for any $r>0$, we define
 \begin{equation}\label{fr}
 	f_r:=\max_{x\in\mathcal{B}_r}f(x)\quad \mbox{and}\quad \mathcal{I}_r^k:=\left\{i\in[N]:x^{i,k}\in\mathcal{B}_r\right\}.
 \end{equation}
  One important estimation for the following analysis is shown as follows, which is similar to the quantitative Laplace principle in \cite{Fornasier2022}.
 \begin{lemma}\label{le-barx}
 	  Suppose Assumption \ref{ass6} holds with the defined parameters $\{R_0, f_\infty, \mu, \nu\}$, and let $r\in(0, R_0]$ and $q>0$ satisfy $q+f_r-f^*\leq f_\infty$. Then, for any $k\geq 0$, if the indicator set $\mathcal{I}_r^k$ is nonempty, we have
 	\begin{equation*}
 		\left\|\bar{x}^{\star,k}-x^*\right\|\leq\frac{(q+f_r-f^*)^\nu}{\mu}+\frac{\mathrm{e}^{-\beta q}}{\left|\mathcal{I}_r^k\right| }\sum_{i=1}^N\left\|x^{i,k}-x^*\right\|.
 	\end{equation*}
 \end{lemma}
\begin{proof}
	Let us first choose $\tilde{r}=\frac{(q+f_r-f^*)^\nu}{\mu}$. By the assumption $q+f_r-f^*\leq f_\infty$ in the statement, this choice satisfies
	\begin{equation}\label{finftygeq}
		f_\infty\geq\left( \mu\tilde{r}\right)^{1/\nu}
	\end{equation} 
	and furthermore $\tilde{r}\geq r$. Indeed, by Assumption \ref{ass6} with $r\in(0,R_0]$, it implies
	\begin{equation*}
		\tilde{r}\geq \frac{1}{\mu}\left(f_r-f^* \right)^\nu=\max_{x\in\mathcal{B}_r}\frac{1}{\mu}\left( f(x)-f^*\right) ^\nu\geq\max_{x\in\mathcal{B}_r}\left\| x-x^*\right\|=r. 
	\end{equation*}
	 For any $k\geq0$, using the definition of  $\bar{x}^{\star,k}$ in \eqref{weightedstate}, we obtain
	\begin{equation}\label{barx1}
		\begin{split}
			\left\|\bar{x}^{\star,k}-x^*\right\|\leq&\frac{\sum_{i=1}^N\mathrm{e}^{-\beta f(x^{i,k})}\left\|x^{i,k}-x^*\right\|}{\sum_{i=1}^N\mathrm{e}^{-\beta f(x^{i,k})}}\\
			=&\frac{\sum_{i\in\mathcal{I}_{\tilde{r}}^k}\mathrm{e}^{-\beta f(x^{i,k})}\left\|x^{i,k}-x^*\right\|+\sum_{i\notin\mathcal{I}_{\tilde{r}}^k}\mathrm{e}^{-\beta f(x^{i,k})}\left\|x^{i,k}-x^*\right\|}{\sum_{i=1}^N\mathrm{e}^{-\beta f(x^{i,k})}}\\
			\leq&\tilde{r}+\frac{\exp\left( -\beta\inf_{x^{i,k}\notin\mathcal{B}_{\tilde{r}}}f(x^{i,k})\right)\sum_{i\notin\mathcal{I}_{\tilde{r}}^k}\left\|x^{i,k}-x^*\right\|}{\sum_{i=1}^N\mathrm{e}^{-\beta f(x^{i,k})}},
		\end{split}
	\end{equation}
where the last inequality is due to $\left\|x^{i,k}-x^*\right\|\leq\tilde{r}$ for all $i\in\mathcal{I}_{\tilde{r}}^k$. In addition, note that the indicator set $\mathcal{I}_r^k$ is nonempty, then we have
\begin{equation*}
	\sum_{i=1}^N\mathrm{e}^{-\beta f(x^{i,k})}\geq\sum_{i\in\mathcal{I}_r^k}\mathrm{e}^{-\beta f(x^{i,k})}\geq \left|\mathcal{I}_r^k\right|\mathrm{e}^{-\beta f_r}.
\end{equation*}
Substituting this into \eqref{barx1}, we have
\begin{equation}\label{barx2}
	\left\|\bar{x}^{\star,k}-x^*\right\|\leq\tilde{r}+\frac{\exp\left( -\beta\left( \inf_{x^{i,k}\notin\mathcal{B}_{\tilde{r}}}f(x^{i,k})-f_r\right) \right)\sum_{i\notin\mathcal{I}_{\tilde{r}}^k}\left\|x^{i,k}-x^*\right\|}{\left|\mathcal{I}_r^k\right|}.
\end{equation}
Using Assumption \ref{ass6} again, if $\tilde{r}>R_0$, we thus have
\begin{equation*}
	\inf_{x^{i,k}\notin\mathcal{B}_{\tilde{r}}}f(x^{i,k})\geq \inf_{x^{i,k}\notin\mathcal{B}_{R_0}}f(x^{i,k})> f^*+f_\infty;
\end{equation*}
 if $\tilde{r}\leq R_0$, we have 
 \begin{equation*}
 	\left( \inf_{x^{i,k}\notin\mathcal{B}_{\tilde{r}}}f(x^{i,k})-f^*\right) ^\nu\geq \mu \tilde{r}\quad\mbox{or}\quad\inf_{x^{i,k}\notin\mathcal{B}_{\tilde{r}}}f(x^{i,k})> f^*+f_\infty.
 \end{equation*}
Hence, considering \eqref{finftygeq} and the definition of $\tilde{r}$, it holds

\begin{equation*}
	\inf_{x^{i,k}\notin\mathcal{B}_{\tilde{r}}}f(x^{i,k})-f_r\geq\min\left\lbrace f^*+f_\infty, f^*+\left( \mu\tilde{r}\right)^{1/\nu} \right\rbrace-f_r\geq f^*+\left( \mu\tilde{r}\right)^{1/\nu}-f_r =q.
\end{equation*}
Inserting this and the definition of $\tilde{r}$ into \eqref{barx2}, we obtain the final result.
\end{proof}
The following lemma provides an upper bound for the expected mean square distance $\mathbb{E}\left[ W_k\right] $.
\begin{lemma}\label{le-EW}
	Under Assumption \ref{ass1}, for all $i\in[N]$ and $k\geq 0$, the expected mean square distance satisfies
	\begin{equation*}
		\begin{split}
			\mathbb{E}\left[ W_{k+1}\right]\leq&\left( 1+2\lambda^2+2\delta^2-2\lambda\right) \mathbb{E}\left[ W_k\right] +\left(2
			\lambda^2+2\delta^2 \right)\mathbb{E}\left[ \left\| \bar{x}^{\star,k}-x^*\right\|^2 \right]\\
			&+\left(4\lambda^2+4\delta^2+2\sqrt{\lambda^2+\delta^2} \right)\sqrt{\mathbb{E}\left[W_k \right]\mathbb{E} \left[ \left\| \bar{x}^{\star,k}-x^*\right\|^2 \right]}\\
			&+2\alpha_kM_g\sqrt{\mathbb{E}\left[W_k \right]}+2\alpha_k^2M_g^2.
		\end{split}
	\end{equation*}
\end{lemma}
\begin{proof}
	By the definition of $W_k$ in \eqref{Wk}, we have
	\begin{equation}\label{Wk+1}
		\begin{split}
			W_{k+1}=&\frac{1}{N}\sum_{i=1}^N\left\|x^{i,k+1}-x^* \right\|^2 =\frac{1}{N}\sum_{i=1}^N\left\|x^{i,k+1}-x^{i,k}+x^{i,k}-x^* \right\|^2\\
			=& W_{k}+\underbrace{\frac{1}{N}\sum_{i=1}^N\left\|x^{i,k+1}-x^{i,k} \right\|^2}_{=:T^k_1}+\underbrace{\frac{2}{N}\sum_{i=1}^N\left( x^{i,k}-x^*\right)^\top\left(x^{i,k+1}-x^{i,k} \right)}_{=:T^k_2}.
		\end{split}
	\end{equation}
Let $H^k\in\mathbb{R}^{d\times d}$ be the diagonal matrix such that $H^k= {\rm{diag}}\left( \eta^k_1,\ldots,\eta^k_d\right)$. Due to \eqref{updatestate1} and \eqref{updatestate2}, it holds
\begin{equation}\label{ET1}
	\begin{split}
		\mathbb{E}\left[ T^k_1\right]=&\frac{1}{N}\sum_{i=1}^N\mathbb{E}\left[\left\| \left( \lambda I_d+H^k\right) \left(x^{i,k}-\bar{x}^{\star,k} \right)+\alpha_kg^{i,k} \right\|^2 \right] \\
		\leq&\frac{2}{N}\sum_{i=1}^N\mathbb{E}\left[\left\| \left( \lambda I_d+H^k\right) \left(x^{i,k}-\bar{x}^{\star,k} \right)\right\| ^2\right]+2\alpha_k^2M_g^2 \\
		\leq&\frac{2}{N}\sum_{i=1}^N\mathbb{E}\underbrace{\left[\left\| \left( \lambda I_d+H^k\right) \left(x^{i,k}-x^* \right)\right\| ^2 \right]}_{=: Z_1^k}+\frac{2}{N}\sum_{i=1}^N\mathbb{E}\underbrace{\left[\left\| \left( \lambda I_d+H^k\right) \left(\bar{x}^{\star,k}-x^* \right)\right\| ^2\right]}_{=: Z_2^k}\\
		&+\frac{4}{N}\sum_{i=1}^N\mathbb{E}\underbrace{\left[\left\|\left( \lambda I_d+H^k\right) \left(x^{i,k}-x^*\right)\right\|\left\|\left( \lambda I_d+H^k\right)\left(\bar{x}^{\star,k}-x^*\right)\right\| \right]}_{=: Z_3^k }+2\alpha_k^2M_g^2,\\
	\end{split}
\end{equation}
where the first inequality follows from $(a+b)^2\leq2(a^2+b^2)$, $\forall a,b\in\mathbb{R}$ and \eqref{bound}, and the last inequality uses adding and subtracting $x^*$ and Cauchy Schwarz inequality. 
We apply \eqref{eta} and H{\"o}lder's inequality to get
\begin{align*}
	\mathbb{E}\left[ Z_1^k\right]=&\sum_{l=1}^d\mathbb{E}\left[\left(\lambda+\eta^k_l \right)^2  \right] \mathbb{E}\left[ \left(x^{i,k}_l-x^*_l \right)^2 \right] =\left(\lambda^2+\delta^2 \right)\mathbb{E}\left[W_k \right] ,\\
	\mathbb{E}\left[ Z_2^k\right]=&\sum_{l=1}^d\mathbb{E}\left[\left(\lambda+\eta^k_l \right)^2  \right] \mathbb{E}\left[ \left(\bar{x}^{\star,k}_l-x^*_l \right)^2 \right] =\left(\lambda^2+\delta^2 \right)\mathbb{E}\left[\left\|\bar{x}^{\star,k}-x^* \right\|^2 \right],\\
	\mathbb{E}\left[ Z_3^k\right]\leq& \sqrt{\sum_{l=1}^d\mathbb{E}\left[ \left(\lambda+\eta^k_l\right)^2\right]\mathbb{E}\left[ \left(x^{i,k}_l-x^*_l\right)^2\right]} \sqrt{ \sum_{l=1}^d\mathbb{E}\left[ \left(\lambda+\eta^k_l\right)^2\right]\mathbb{E}\left[ \left(\bar{x}^{\star,k}_l-x^*_l\right)^2\right]} \\
	=&\left(\lambda^2+\delta^2 \right)\sqrt{\mathbb{E}\left[\left\|x^{i,k}-x^* \right\|^2  \right]} \sqrt{\mathbb{E}\left[\left\|\bar{x}^{\star,k}-x^* \right\|^2  \right]}.
\end{align*}
Thus, we plug these relations back to \eqref{ET1}, and obtain
\begin{equation}\label{ET11}
	\begin{split}
		\mathbb{E}\left[T_1^k \right]\leq& 2\left( \lambda^2+\delta^2\right) \left( \mathbb{E}\left[W_k\right]+2\sqrt{\mathbb{E}\left[W_k \right]}\sqrt{\mathbb{E}\left[\left\|\bar{x}^{\star,k}-x^* \right\|^2  \right]}+\mathbb{E}\left[\left\|\bar{x}^{\star,k}-x^* \right\|^2  \right]\right) \\
		&+2\alpha_k^2M_g^2.
	\end{split}
\end{equation}
Similarly, we have
\begin{equation}\label{T2}
	\begin{split}
		T^k_2=&\frac{2}{N}\sum_{i=1}^N\left(x^{i,k}-x^*\right)^\top\left(-\left(\lambda I_d+H^k\right) \left(x^{i,k}-x^*+x^*-\bar{x}^{\star,k} \right)-\alpha_kg^{i,k}\right)\\
		\leq&-\frac{2}{N}\sum_{i=1}^N\sum_{l=1}^d\left(\lambda+\eta^k_l\right)\left(x^{i,k}_l-x^*_l\right)^2+\frac{2}{N}\sum_{i=1}^N\alpha_k\left\| g^{i,k}\right\| \left\|x^{i,k}-x^* \right\| \\
		&+\frac{2}{N}\sum_{i=1}^N \left\| x^{i,k}-x^*\right\| \left\|    \left(\lambda I_d+H^k\right) \left(\bar{x}^{\star,k}-x^* \right)\right\| . 
	\end{split}
\end{equation}
Taking expectation in \eqref{T2}, by H{\"o}lder's inequality and \eqref{bound}, it holds that
\begin{equation}\label{ET2}
	\begin{split}
		\mathbb{E}\left[T_2^k\right]\leq&-2\lambda\mathbb{E}\left[W_k\right]+\frac{2\alpha_kM_g}{N}\sum_{i=1}^N\sqrt{\mathbb{E}\left[\left\| x^{i,k}-x^*\right\|^2\right]}\\
		&+\frac{2}{N}\sum_{i=1}^N\sqrt{\mathbb{E}\left[\left\|x^{i,k}-x^*\right\|^2\right]}\sqrt{\lambda^2+\delta^2}\sqrt{\mathbb{E}\left[\left\|\bar{x}^{\star,k}-x^*\right\|^2\right]}.
	\end{split}
\end{equation}
Moreover, by Cauchy-Schwarz inequality, one has
\begin{equation*}
	\frac{1}{N}\sum_{i=1}^N\sqrt{\mathbb{E}\left[\left\|x^{i,k}-x^*\right\|^2\right]}\leq\frac{1}{\sqrt{N}}\sqrt{\sum_{i=1}^N\mathbb{E}\left[\left\|x^{i,k}-x^*\right\|^2\right]}=\sqrt{\mathbb{E}\left[W_k\right]},
\end{equation*}
which along with \eqref{ET2} leads to
\begin{equation}\label{ET22}
	\begin{split}
		\mathbb{E}\left[T_2^k\right]\leq&-2\lambda\mathbb{E}\left[W_k\right]+2\alpha_kM_g\sqrt{\mathbb{E}\left[W_k\right]}\\
		&+2\sqrt{\lambda^2+\delta^2}\sqrt{\mathbb{E}\left[W_k\right]}\sqrt{\mathbb{E}\left[\left\|\bar{x}^{\star,k}-x^*\right\|^2\right]}.
	\end{split}
\end{equation}
Finally, combining \eqref{Wk+1}, \eqref{ET11} and \eqref{ET22}, we get the desired result.
\end{proof}
In the next theorem, we provide the iteration complexity in expectation for the ESCBO algorithm. For a given $\xi\in(0,1)$, we first define 
\begin{equation}\label{define}
	\begin{split}
		&\gamma:=1-(1-\xi)\left(2\lambda-2\lambda^2-2\delta^2 \right) ,\\
		&\kappa:=\min\left\{\frac{\xi\left( 2\lambda-2\lambda^2-2\delta^2\right) }{4\left( 2\lambda^2+2\delta^2+\sqrt{ \lambda^2+\delta^2}+1 \right) },\sqrt{ \frac{\xi\left( 2\lambda-2\lambda^2-2\delta^2\right)}{2\left(2\lambda^2+2\delta^2+2 \right) }}\right\},\\
		&C_{4,k}:=\kappa\sqrt{\mathbb{E}\left[ W_k\right]}.
	\end{split}
\end{equation}
And using the constants given in Assumption \ref{ass6} and \eqref{define}, for any $k\geq 0$, we set
\begin{equation}\label{qr}
	\begin{split}
		q_k&:=\frac{1}{2}\min\left\lbrace f_\infty, \left( \frac{\mu C_{4,k}}{\sqrt{2}}\right)^{1/\nu} \right\rbrace, \\
		r_k&:=\max\left\lbrace s\in(0,R_0]: f_s(x)-f^*\leq q_k\right\rbrace,
	\end{split}
\end{equation}
where $f_s$ is defined as in \eqref{fr}. For a given $\varepsilon>0$, set
\begin{equation}\label{K}
	\begin{split}
		&K:=\sup\left\lbrace k> 0: \mathbb{E}\left[ W_{k'}\right]>\varepsilon ~\mbox{and}~ \mathbb{E}\left[\left\| \bar{x}^{\star,k'}-x^*\right\|^2 \right] \leq C_{4,k'}^2, ~\forall k'\in[0,k)\right\rbrace,\\
		&K_\varepsilon:=\left\lceil \log_{1/\gamma}\frac{\mathbb{E}\left[W_0\right]}{\varepsilon} \right\rceil. 
	\end{split}
\end{equation}
\begin{theorem}
	Under the Assumptions \ref{ass1} and \ref{ass6}, for any given $\varepsilon\in\left( 0,\mathbb{E}\left[W_0\right]\right) $, let $K$ and $K_\varepsilon$ be defined in \eqref{K}, and assume that the parameters $\{\lambda, \delta, \{\alpha_k\}_{k\geq0}\}$ in the ESCBO algorithm satisfy 
	\begin{equation}\label{condition5}
		2\lambda-2\lambda^2-2\delta^2>0,\quad 0<\gamma<1,\quad \alpha_kM_g\leq\kappa\sqrt{\varepsilon},\quad \forall k< K.
	\end{equation}
Then, if $\mathcal{I}_{r_K}^K$ is nonempty, there exists a $\beta_0>0$ such that for all $\beta>\beta_0$, the ESCBO algorithm satisfies
\begin{equation*}
	\min_{k\in[0,K_\varepsilon]}\mathbb{E}\left[ W_{k}\right]\leq\varepsilon,
\end{equation*}  
and it holds $\mathbb{E}\left[ W_{k}\right]<\gamma^k\mathbb{E}\left[ W_{0}\right]$ until $\mathbb{E}\left[ W_{k}\right]$ reaches the prescribed accuracy $\varepsilon$.
\end{theorem}
\begin{proof}
First, we claim that $K>0$. Indeed, it suffices to prove that $\mathbb{E}\left[W_0 \right]>\varepsilon$ and $\mathbb{E}\left[\left\| \bar{x}^{\star,0}-x^*\right\|^2 \right] \leq C_{4,0}^2$. The former is immediate by the assumption. From \eqref{qr}, we have $r_0\leq R_0$ and $q_0+f_{r_0}-f^*\leq 2q_0\leq f_\infty$. Then by Lemma \ref{le-barx} and \eqref{qr}, we have
	\begin{equation*}
		\begin{split}
			\left\|\bar{x}^{\star,0}-x^*\right\|\leq&\frac{(q_0+f_{r_0}-f^*)^\nu}{\mu}+\frac{\mathrm{e}^{-\beta q_0}}{\left|\mathcal{I}_{r_0}^0\right| }\sum_{i=1}^N\left\|x^{i,0}-x^*\right\|\\
			\leq&\frac{C_{4,0}}{\sqrt{2}}+\frac{\mathrm{e}^{-\beta q_0}}{\left|\mathcal{I}_{r_0}^0\right| }\sum_{i=1}^N\left\|x^{i,0}-x^*\right\|.
		\end{split}		
	\end{equation*}
Therefore, we can find a $\tilde{\beta}>0$ so that for any $\beta>\tilde{\beta}$, $\left\|\bar{x}^{\star,0}-x^*\right\|\leq C_{4,0}$, which gives the latter.

	By Lemma \ref{le-EW}, for all $k\in[0,K)$, we have 
	\begin{equation}\label{EWk+1}
		\begin{split}
			&\mathbb{E}\left[ W_{k+1}\right]\\
			<&\left( 1+2\lambda^2+2\delta^2-2\lambda\right) \mathbb{E}\left[ W_k\right] +\left(2
			\lambda^2+2\delta^2 \right)C_{4,k}^2+2C_{4,k}^2\\
			&+\left(4\lambda^2+4\delta^2+\sqrt{\lambda^2+\delta^2} \right)C_{4,k}\sqrt{\mathbb{E}\left[W_k \right]}
			+2C_{4,k}\sqrt{\mathbb{E}\left[W_k \right]}\\
			=&\left( 1+2\lambda^2+2\delta^2-2\lambda+\left(2
			\lambda^2+2\delta^2+2\right)\kappa^2
			+2\left(2\lambda^2+2\delta^2+\sqrt{\lambda^2+\delta^2}+1 \right)\kappa\right) \mathbb{E}\left[W_k \right]\\
			\leq&\left( 1-(1-\xi)\left(2\lambda-2\lambda^2-2\delta^2 \right) \right)\mathbb{E}\left[W_k \right]\\
			=&\gamma\mathbb{E}\left[W_k \right]\leq\gamma^{k+1}\mathbb{E}\left[W_0 \right],
		\end{split}
	\end{equation}
where the first inequality uses the definition of $K$ in \eqref{K} and $\alpha_kM_g\leq\kappa\sqrt{\varepsilon}<\kappa\sqrt{\mathbb{E}[W_k]}= C_{4,k}$ by \eqref{condition5}, the first equality, the second inequality and the last equality all follow from \eqref{define}.
To derive the conclusion, it remains to check three cases:

\textbf{Case 1}: $K_\varepsilon\leq K$. By \eqref{K} and \eqref{EWk+1}, for all $k\in[0,K_\varepsilon]$, we have $\mathbb{E}\left[ W_{k}\right]<\gamma^{k}\mathbb{E}\left[W_0 \right]$, which together with the definition of $K_\varepsilon$ in \eqref{K} yields that $\mathbb{E}\left[ W_{K_\varepsilon}\right]\leq\varepsilon$. 

\textbf{Case 2}: $K_\varepsilon> K$ and $\mathbb{E}\left[ W_{K}\right]\leq\varepsilon$. In this case, the conclusion is clearly valid.

\textbf{Case 3}: $K_\varepsilon> K$, $\mathbb{E}\left[ W_{K}\right]>\varepsilon$ and $\mathbb{E}\left[\left\| \bar{x}^{\star,K}-x^*\right\|^2 \right] > C_{4,K}^2$. We claim that if $\mathbb{E}\left[ W_{K}\right]>\varepsilon$, there exists a $\beta_0>0$ so that for all $\beta\geq\beta_0$, it holds
\begin{equation}\label{contradict}
	\mathbb{E}\left[\left\| \bar{x}^{\star,K}-x^*\right\|^2 \right] \leq C_{4,K}^2,
\end{equation}
which leads to a contradiction to $\mathbb{E}\left[\left\| \bar{x}^{\star,K}-x^*\right\|^2 \right] > C_{4,K}^2$ and proves that \textbf{Case 3} cannot occur.
To end the proof, we shall show that \eqref{contradict} holds. Indeed, by Cauchy-Schwarz inequality, one has
\begin{equation*}
	\frac{1}{N}\sum_{i=1}^N\left\|x^{i,K}-x^*\right\|\leq\frac{1}{\sqrt{N}}\sqrt{\sum_{i=1}^N\left\|x^{i,K}-x^*\right\|^2}=\sqrt{W_K}.
\end{equation*} 
Moreover, from \eqref{qr}, we have $r_K\leq R_0$ and $q_K+f_{r_K}-f^*\leq 2q_K\leq f_\infty$.
We can apply Lemma \ref{le-barx} with $q_K$ and $r_K$ defined in \eqref{qr} to get
\begin{equation*}
	\left\|\bar{x}^{\star,K}-x^*\right\|\leq\frac{(q_K+f_{r_K}-f^*)^\nu}{\mu}+\frac{\mathrm{e}^{-\beta q_K}}{\left|\mathcal{I}_{r_K}^K\right| }\sum_{i=1}^N\left\|x^{i,K}-x^*\right\|\leq\frac{C_{4,K}}{\sqrt{2}}+N\mathrm{e}^{-\beta q_K}\sqrt{W_K}.
\end{equation*}
Then, we have
\begin{equation}\label{barxK}
	\mathbb{E}\left[\left\|\bar{x}^{\star,K}-x^*\right\|^2\right]\leq\frac{C_{4,K}^2}{2}+\sqrt{2}N\mathrm{e}^{-\beta q_K}\mathbb{E}\left[\sqrt{W_K}\right]+N^2\mathrm{e}^{-2\beta q_K}\mathbb{E}\left[W_K\right].
\end{equation}
Therefore, we can find a $\beta_0>\tilde{\beta}>0$ so that for any $\beta>\beta_0$, 
\begin{equation*}
	\sqrt{2}N\mathrm{e}^{-\beta q_K}\mathbb{E}\left[\sqrt{W_K}\right]+N^2\mathrm{e}^{-2\beta q_K}\mathbb{E}\left[W_K\right]\leq\frac{C_{4,K}^2}{2}, 
\end{equation*}
which together with \eqref{barxK} gives the contradiction \eqref{contradict} and concludes the proof.
\end{proof}
\section{Numerical experiments}\label{sec7}
In this section, to illustrate the numerical performance of the ESCBO algorithm, we apply it to minimize some nonconvex benchmark functions and train a class of deep neural networks.  
\subsection{Validation of global consensus}\label{sec7.1}
In this part, we verify the emergence of global consensus of the ESCBO algorithm by minimizing the following typical benchmark function, namely Rastrigin \cite{Storn1997}:
\begin{equation}\label{f1}
	f(x)=\frac{1}{d}\sum_{l=1}^d\left(x_l^2-10\cos(2\pi x_l)+10\right),
\end{equation}
where $x=\left( x_1,\ldots,x_d\right)^\top\in\mathbb{R}^d$. 
Rastrigin function has the unique global minimizer $x^*=(0,\ldots,0)^\top\in\mathbb{R}^d$ with $f(x^*)=0$. It is known that for $x\in(-5,5)^d$, the number of its local minimizers is $10^d$, which is related to dimensionality. It is highly nonconvex and multimodal.
We perform the ESCBO algorithm for minimizing \eqref{f1} with $d=2$. Initial points $x^{i,0}$, $i\in[N]$ are uniformly sampled from $[-5,5]^2$. The parameters in the ESCBO algorithm are chosen as follows:
\begin{equation*}
	N=20,\quad\lambda=0.01, \quad \delta=0.1, \quad \beta=100,\quad \sigma=0.0001,\quad \alpha_k=\frac{1}{2k}.
\end{equation*}
We terminate the ESCBO algorithm once $x^{i,k+1}$ and $x^{i,k}$ satisfy
\begin{equation}\label{stop1}
	\max_{i\in[N]} \left\| x^{i,k+1}-x^{i,k}\right\| \leq 1\mathrm{e}\mbox{-}06 \quad   \mbox{and} \quad \max_{i\in[N]} \frac{\left| f(x^{i,k+1})-f(x^{i,k})\right| }{\left\| x^{i,k+1}-x^{i,k}\right\|} \leq 1\mathrm{e}\mbox{-}06.
\end{equation}

The output of our algorithm is $x^{1,k}=x^{2,k}=(-5.00\mathrm{e}$-$05, -5.00\mathrm{e}$-$05)^\top$ when $k\geq536$. To visualize the existence of global consensus in practice, we depict in Fig. \ref{fig1} the positions of the particles at different iterations $k\in\{0, 50, 150, 536\}$, respectively. From Fig. \ref{fig1}, we can see that the particles ultimately exhibit global consensus and successfully find the approximate global minimizer of Rastrigin function in \eqref{f1}. We remark in passing that although we assume the condition \eqref{condition1} to guarantee the global consensus of the ESCBO algorithm, the simulation results suggest that this condition is actually not essential for this problem. 

\begin{figure}[h]
	\vspace*{-5mm}
	\centering
	\includegraphics[width=0.95\linewidth]{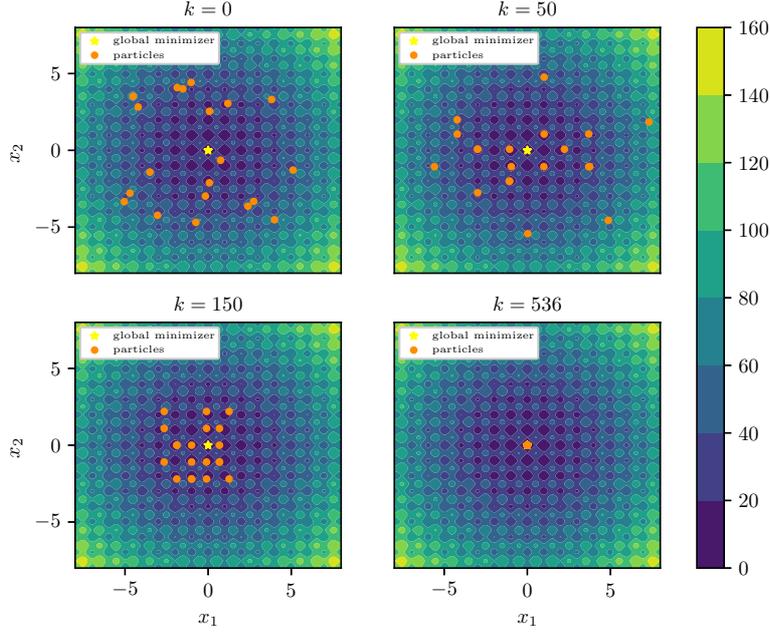}
	\caption{Positions of particles at some different iterations $k$ in Section \ref{sec7.1} }
	\label{fig1}
\end{figure}

\subsection{Numerical comparisons with the vanilla-CBO algorithm}\label{sec7.2}
In what follows, we numerically test the excellent performance of the ESCBO algorithm. We compare the ESCBO algorithm with the vanilla-CBO algorithm in \cite{Ha2021} for several nonconvex benchmark problems in optimization, including Rastrigin function in \eqref{f1} and other six functions, whose formulas and global minimizers are summarized in Table \ref{function}. These objective functions are introduced in \cite{Ali2005,Jamil2013} and the corresponding 2D plots are shown in Fig. \ref{2Dfig}.
\begin{table}[h!]
	\caption{Considered benchmark functions in Section \ref{sec7.2}}
	\centering
	\renewcommand\arraystretch{2}
		\begin{tabular}{lll} 
			\hline
			\multirow{2}{*}{Benchmark}& \multirow{2}{*}{Formula} &  Global minimum\\
			&&and minimizer\\
			\hline 
			\hline
			Salomon&$1-\cos\left (2\pi\sqrt{\sum_{l=1}^dx_l^2}\right)+0.1\sqrt{\sum_{l=1}^dx_l^2}$&$0$ at $(0,\ldots,0)$\\
			\hline
			Griewank&$1+\sum_{l=1}^d\frac{x_l^2}{4000}-\prod_{l=1}^d\cos\left(\frac{x_l}{\sqrt{l}} \right) $&$0$ at $(0,\ldots,0)$\\
			\hline
			\multirow{2}{*}{Ackley} &$-20\exp\left (-0.2\sqrt{\frac{1}{d}\sum_{l=1}^{d}x_l^2}\right )-\exp\left (\frac{1}{d}\sum_{l=1}^{d}\cos(2\pi x_l)\right )$& \multirow{2}{*}{$0$ at $(0,\ldots,0)$}\\ 
			&$+ 20 + \exp(1)$& \\		
			\hline
			Xin-She Yang $4$&$\left(\sum_{l=1}^d \sin^2(x_l)-\mathrm{e}^{-\sum_{l=1}^dx_l^2}\right) \mathrm{e}^{-\sum_{l=1}^d\sin^2\sqrt{|x_l|} }+1 $&$0$ at $(0,\ldots,0)$\\
			\hline
			Bartels Conn& $\left|x_1^2+x_2^2+x_1x_2 \right|+\left|\sin(x_1) \right|+\left| \cos(x_2)\right|  $ &$1$ at $(0,0)$\\
			\hline
			\multirow{3}{*}{Schaffer $4$}&\multirow{3}{*}{$0.5+\frac{\cos^2\left( \sin\left( \left|x_1^2-x_2^2 \right| \right) \right)-0.5 }{\left( 1+0.001\left(x_1^2+x_2^2 \right) \right)^2 }$}&$0.292579$ at \\
			&&$(0,\pm 1.253115)$,\\
			&&$(\pm 1.253115,0)$\\
			\hline
		\end{tabular}
	\label{function}
\end{table}

\begin{figure}[H]
	\centering
	\vspace{-5mm}
	\subfigtopskip=-4pt
	\subfigbottomskip=-1pt
	\subfigcapskip=-7pt
	\setlength{\abovecaptionskip}{17pt}
	\subfigure[Rastrigin]{
		\includegraphics[width=0.4\linewidth]{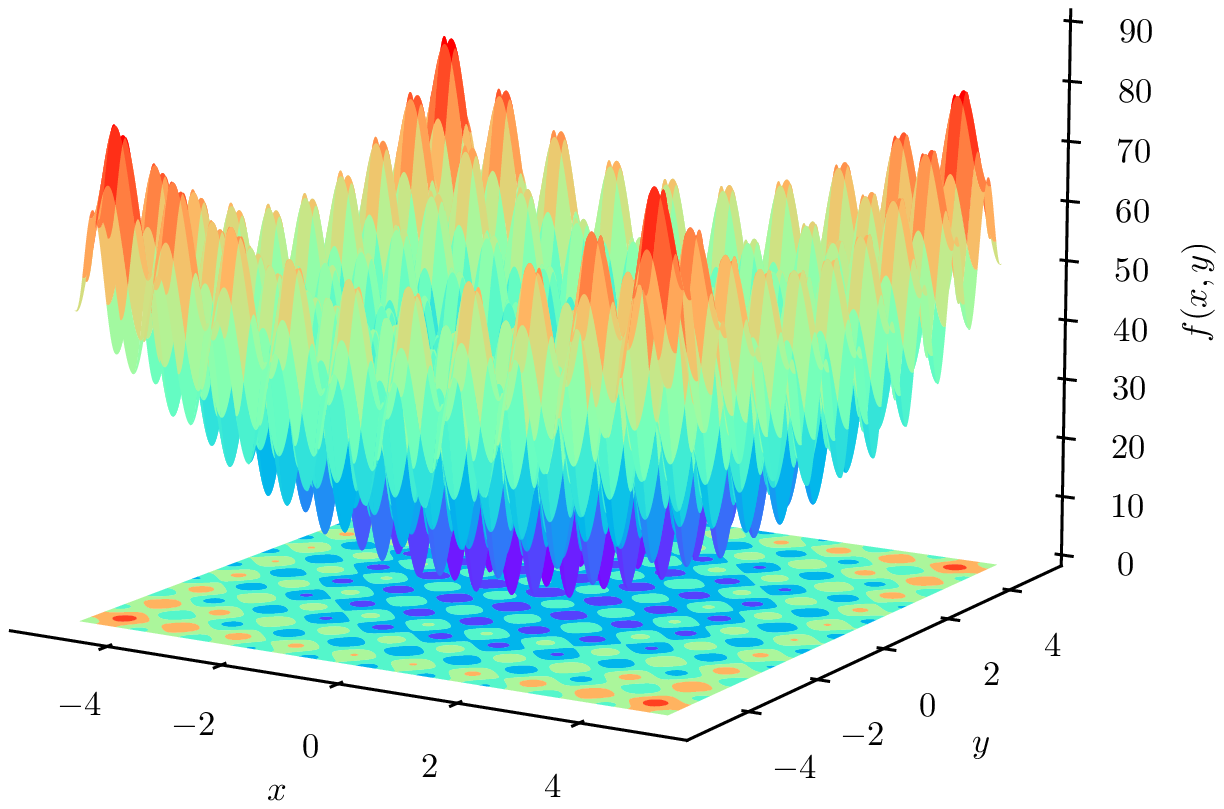}
	}
	
	\subfigure[Salomon]{
		\includegraphics[width=0.4\linewidth]{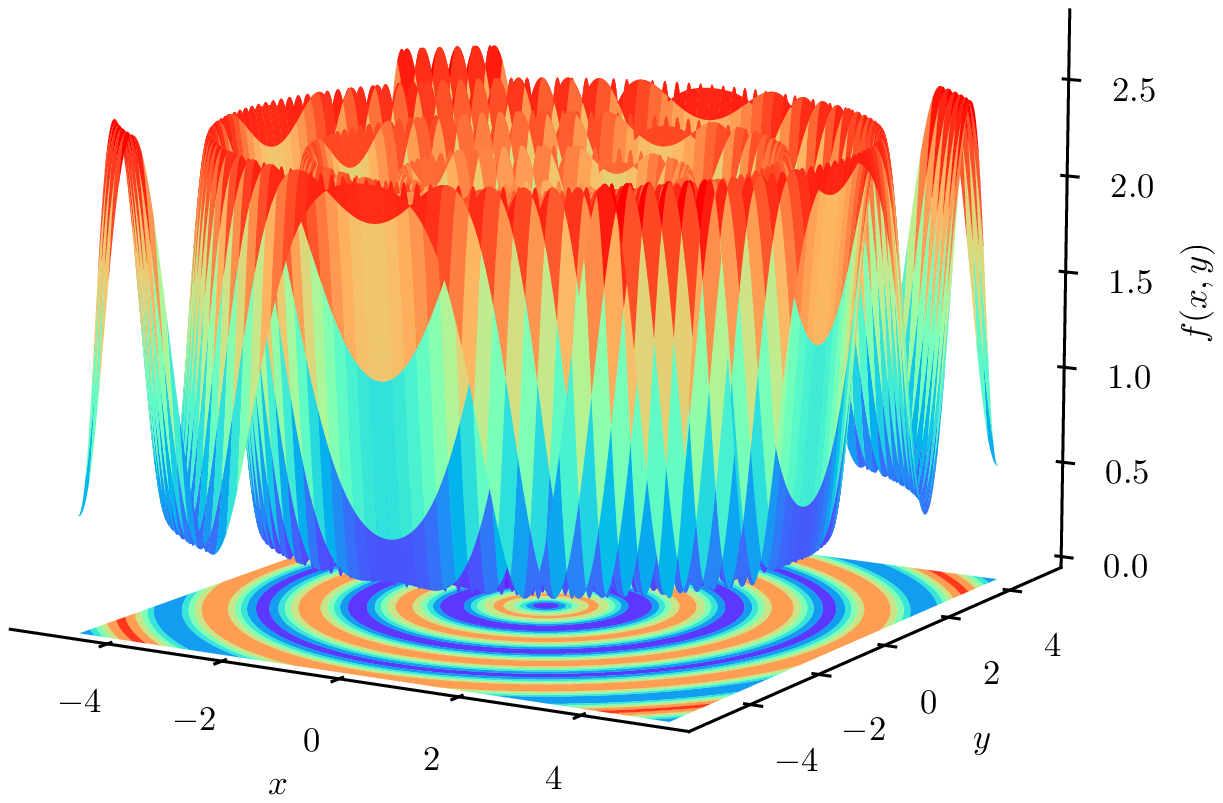}
	}
	\subfigure[Griewank]{
		\includegraphics[width=0.4\linewidth]{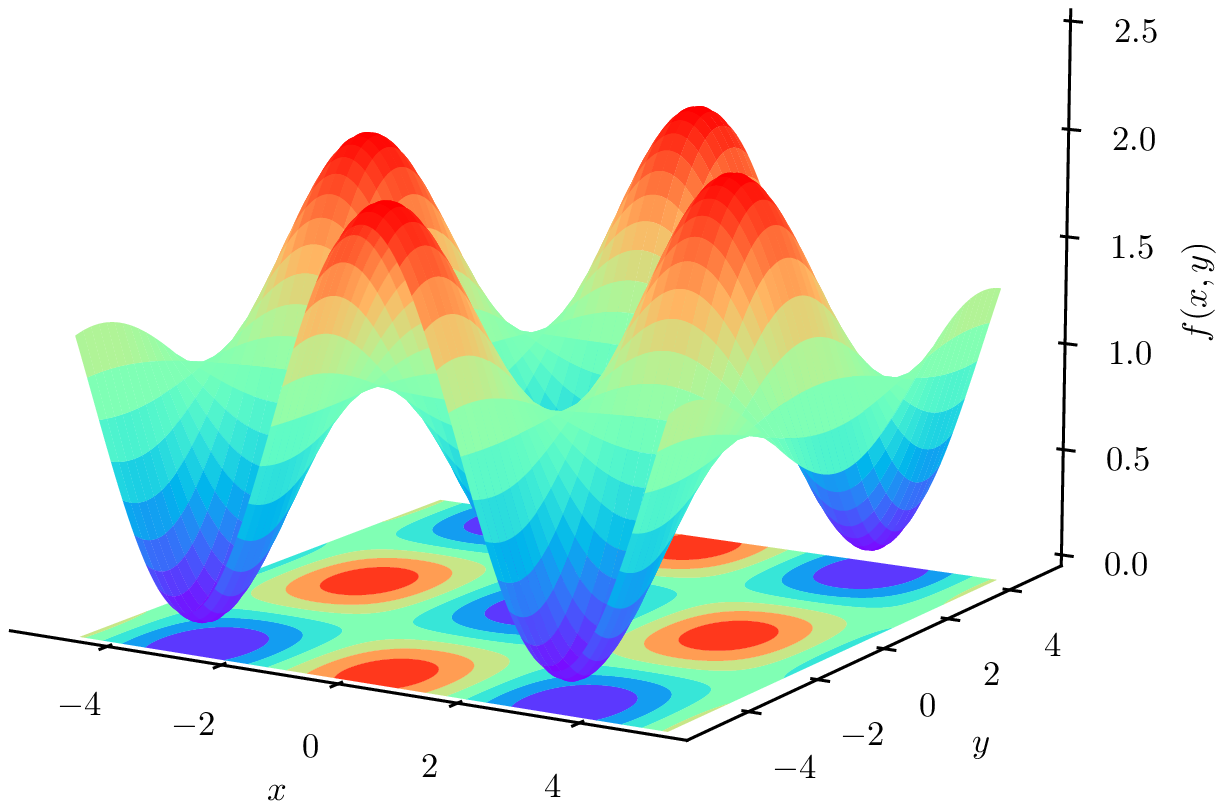}
	}
	\subfigure[Ackley]{
		\includegraphics[width=0.4\linewidth]{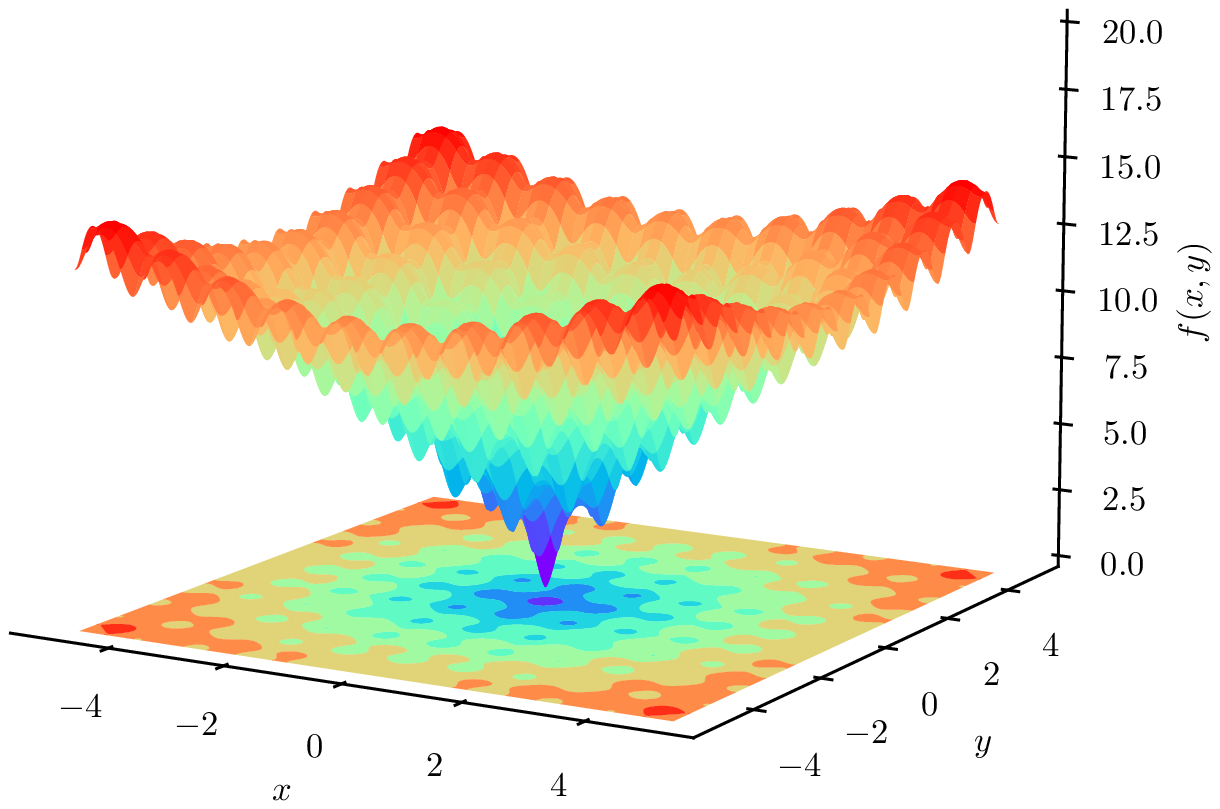}
	}
	\subfigure[Xin-She Yang $4$]{
		\includegraphics[width=0.4\linewidth]{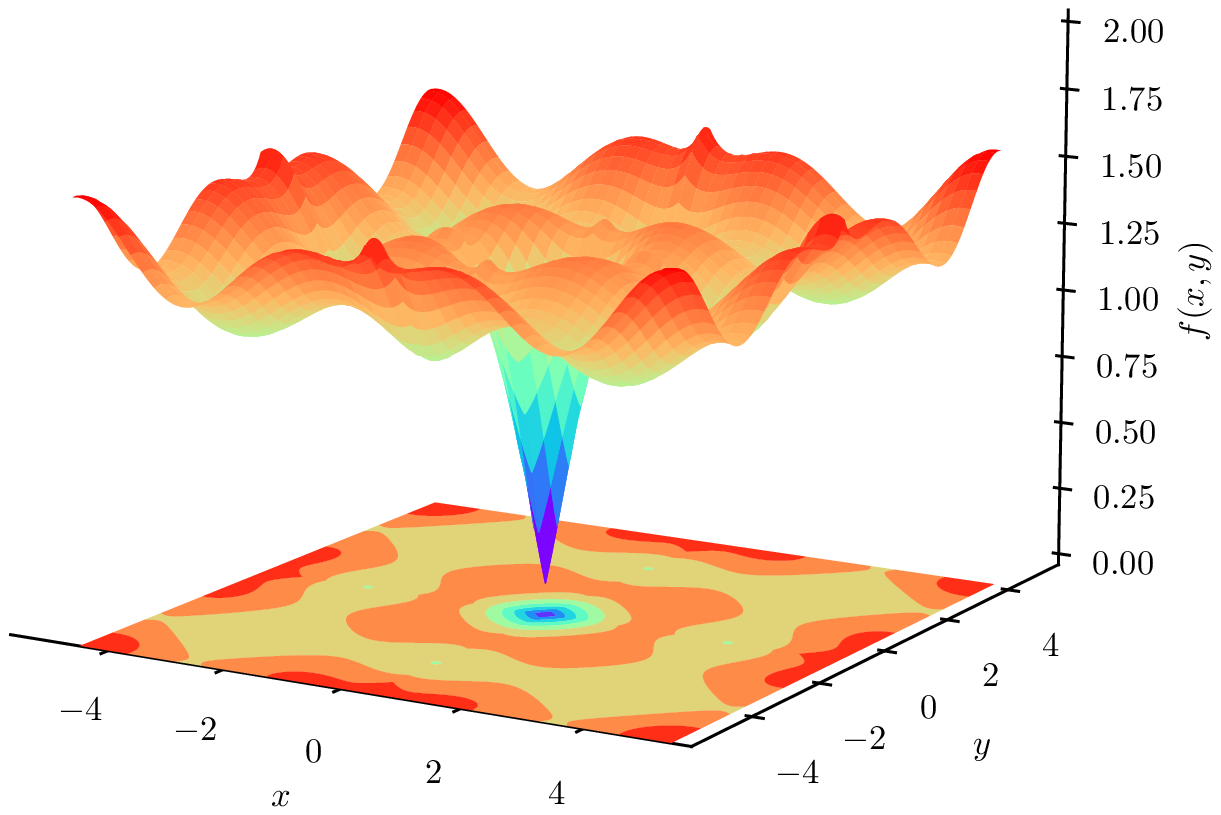}
	}
	\subfigure[Bartels Conn]{
		\includegraphics[width=0.4\linewidth]{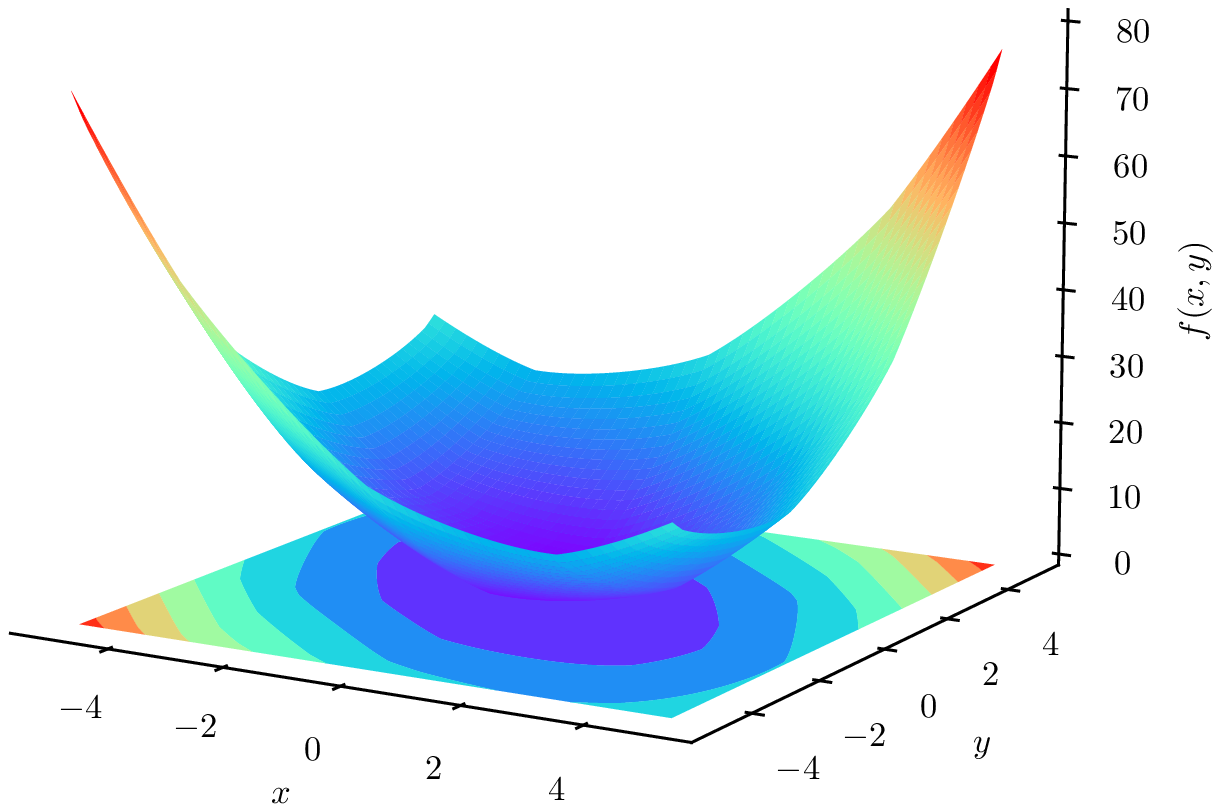}
	}
	\subfigure[Schaffer $4$]{
		\includegraphics[width=0.4\linewidth]{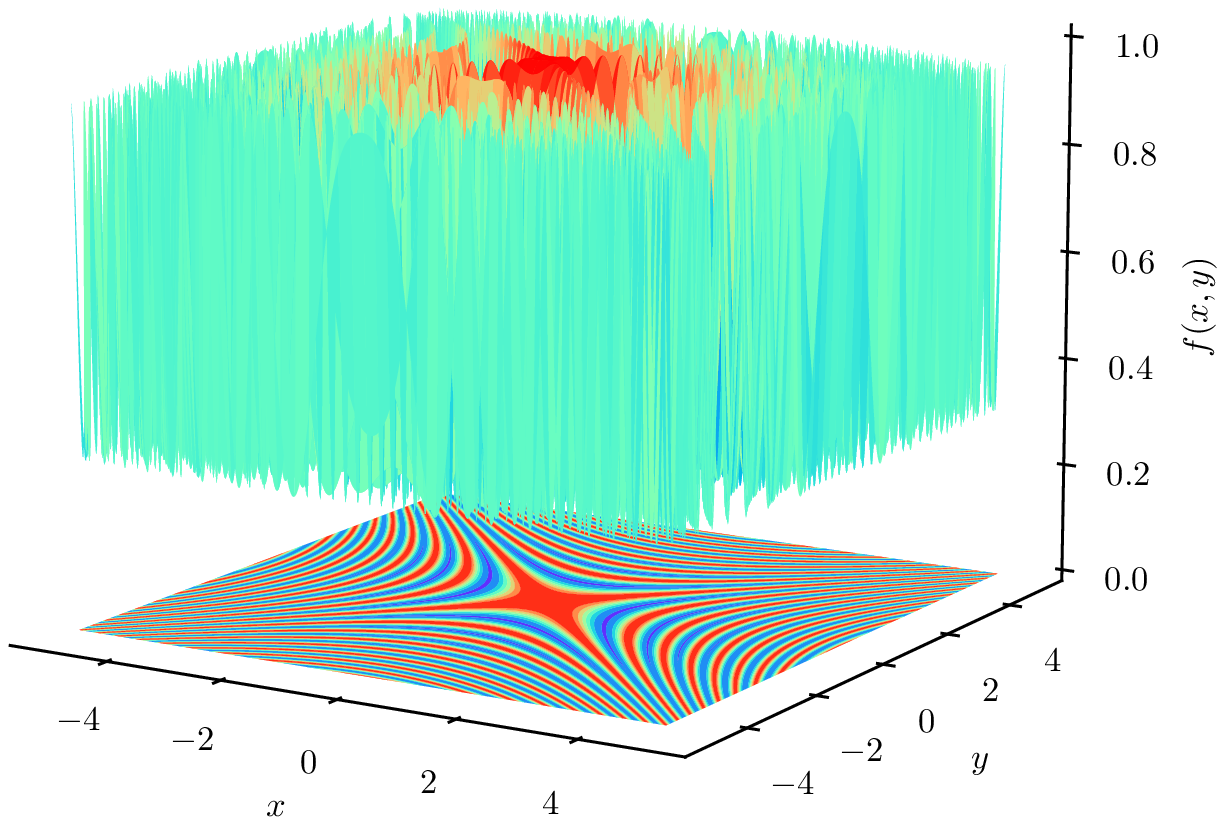}
	}
	\caption{2D plots of benchmark functions in Sections \ref{sec7.1} and \ref{sec7.2} with $d=2$.} 
	\label{2Dfig}
\end{figure}

In both algorithms, we set the parameters $\lambda=0.01$, $\delta=0.1$ and $\beta=1\mathrm{e}$+$20$. We additionally set the step size $\alpha_k=0.99^k$ and the finite difference parameter $\sigma=1\mathrm{e}$-$05$ in the ESCBO algorithm. The number of particles, the initial data and the dimension for each experiment are specified hereafter. The stopping criterion is set when the relation \eqref{stop1} is satisfied or the iteration repeats $10,000$ times. And the success criterion is defined as achieving
\begin{equation}\label{success}
	\max_{i\in[N]}\left\| x^{i,k}-x^*\right\| <1\mathrm{e} \mbox{-}03,
\end{equation}
which ensures that the algorithm obtains a solution quite close to global minimizer $x^*$. Based on \eqref{success}, we record the success rate (rate) to compare the accuracy of both algorithms.
In addition, two more performance measures are adopted, where the first is the averaged solution errors (sol-err) defined by $\frac{1}{N}\sum_{i=1}^N\left\| x^{i,k}-x^*\right\|^2$, and the other measurement is the averaged objective function value errors (fun-err) calculated by $\frac{1}{N}\sum_{i=1}^N \left| f(x^{i,k})-f(x^*)\right|$. These results are averaged over 100 runs of each instance.

Table \ref{com1} records the results of both algorithms for minimizing various benchmark functions mentioned above with different values of $N\in\{20d,40d,60d\}$. The initial data $x_0^i$, $i\in[N]$ are distributed following a uniform distribution in $[-5,5]^d$. As shown, both algorithms have higher success rates as the number of particles increases in general. Notably, the ESCBO algorithm outperforms the vanilla-CBO algorithm for these objective functions, and even for a small number of particles, the global minimizer of the problem can be found by the ESCBO algorithm with high success rates. Furthermore, there are cases for which the vanilla-CBO algorithm fails but the ESCBO algorithm succeeds, especially for the cases that the dimension of the function is large. 

In Table \ref{com2}, we report the results comparing the ESCBO algorithm and the vanilla-CBO algorithm with $N=120$ for different initial datas. Here, $\mathcal{U}(-3,3)$ and $\mathcal{U}(2,6)$ denote that the initial particles are distributed following a uniform distribution in $[-3,3]^d$ and $[2,6]^d$, respectively, while $\mathcal{N}_d(0,3)$ signifies the component of each initial point is drawn from a normal distribution $\mathcal{N}(0,3)$. It is observed that when the initial particles are situated close to the global minimizer, performance of both methods can be effectively enhanced. In the case of $\mathcal{U}(2,6)$, most initial particles are far away from the global minimizer, thereby it results in the failure of the vanilla-CBO algorithm. However, the ESCBO algorithm maintains a success rate of close to $100\%$ for these cases, which indicates that the ESCBO algorithm is less sensitive to the initial data.

In conclusion, these results demonstrate that the performance of the ESCBO algorithm is significantly better compared to the vanilla-CBO algorithm for finding the global minimizer of nonconvex functions.
\begin{table*}[!h]
	\caption{Comparisons between ESCBO and vanilla-CBO with different values of $N$} \label{com1}
	\begin{center}
		\renewcommand\arraystretch{1.2}
		\resizebox{\textwidth}{!}{
			\begin{tabular}{llllllll} \hline
				&&\multicolumn {3}{c}{ESCBO}&\multicolumn {3}{c}{vanilla-CBO}\\
				\cmidrule(lr{0pt}){3-5}\cmidrule(lr{0pt}){6-8}
				&&$N=20d$&$N=40d$&$N=60d$&$N=20d$&$N=40d$&$N=60d$\\
				\hline
				Rastrigin&rate&$\textbf{100\%}$&$\textbf{100\%}$&$\textbf{100\%}$&$12\%$&$29\%$&$37\%$\\
				$(d=3)$&sol-err&$\textbf{2.93}\mathrm{\textbf{e}}$\bf{-}$\textbf{13}$
				&$\textbf{2.23}\mathrm{\textbf{e}}$\bf{-}$\textbf{13}$&$\textbf{2.33}\mathrm{\textbf{e}}$\bf{-}$\textbf{13}$&$5.05\mathrm{e}$-$01$&$4.06\mathrm{e}$-$01$&$1.89\mathrm{e}$-$01$\\
				~&fun-err&$\textbf{1.94}\mathrm{\textbf{e}}$\bf{-}$\textbf{11}$
				&$\textbf{1.48}\mathrm{\textbf{e}}$\bf{-}$\textbf{11}$&$\textbf{1.54}\mathrm{\textbf{e}}$\bf{-}$\textbf{11}$&$5.81\mathrm{e}$-$01$&$1.95\mathrm{e}$-$01$&$1.58\mathrm{e}$-$01$\\
				\hline
				Rastrigin&rate&$\textbf{57\%}$&$\textbf{88\%}$&$\textbf{100\%}$&$0\%$&$0\%$&$0\%$\\
				$(d=10)$&sol-err&$\textbf{4.95}\mathrm{\textbf{e}}$\bf{-}$\textbf{01}$
				&$\textbf{1.98}\mathrm{\textbf{e}}$\bf{-}$\textbf{01}$&$\textbf{2.80}\mathrm{\textbf{e}}$\bf{-}$\textbf{10}$&$1.04\mathrm{e}$+$01$&$9.62\mathrm{e}$+$00$&$7.46\mathrm{e}$+$00$\\
				~&fun-err&$\textbf{4.97}\mathrm{\textbf{e}}$\bf{-}$\textbf{02}$
				&$\textbf{1.99}\mathrm{\textbf{e}}$\bf{-}$\textbf{02}$&$\textbf{5.56}\mathrm{\textbf{e}}$\bf{-}$\textbf{09}$&$2.55\mathrm{e}$+$00$&$1.92\mathrm{e}$+$00$&$1.55\mathrm{e}$+$00$\\
				\hline
				Salomon&rate&$\textbf{100\%}$&$\textbf{100\%}$&$\textbf{100\%}$&$11\%$&$13\%$&$37\%$\\
				$(d=3)$&sol-err&$\textbf{7.17}\mathrm{\textbf{e}}$\bf{-}$\textbf{13}$
				&$\textbf{7.44}\mathrm{\textbf{e}}$\bf{-}$\textbf{13}$&$\textbf{6.56}\mathrm{\textbf{e}}$\bf{-}$\textbf{13}$&$5.97\mathrm{e}$-$01$&$6.96\mathrm{e}$-$01$&$2.99\mathrm{e}$-$01$\\
				~&fun-err&$\textbf{8.44}\mathrm{\textbf{e}}$\bf{-}$\textbf{08}$
				&$\textbf{8.51}\mathrm{\textbf{e}}$\bf{-}$\textbf{08}$&$\textbf{8.03}\mathrm{\textbf{e}}$\bf{-}$\textbf{08}$&$6.96\mathrm{e}$-$02$&$7.04\mathrm{e}$-$02$&$3.90\mathrm{e}$-$02$\\
				\hline
				Salomon&rate&$\textbf{100\%}$&$\textbf{100\%}$&$\textbf{100\%}$&$0\%$&$0\%$&$0\%$\\
				$(d=10)$&sol-err&$\textbf{2.44}\mathrm{\textbf{e}}$\bf{-}$\textbf{12}$
				&$\textbf{2.29}\mathrm{\textbf{e}}$\bf{-}$\textbf{12}$&$\textbf{2.20}\mathrm{\textbf{e}}$\bf{-}$\textbf{12}$&$3.89\mathrm{e}$+$00$&$3.59\mathrm{e}$+$00$&$2.19\mathrm{e}$+$00$\\
				~&fun-err&$\textbf{1.52}\mathrm{\textbf{e}}$\bf{-}$\textbf{07}$
				&$\textbf{1.48}\mathrm{\textbf{e}}$\bf{-}$\textbf{07}$&$\textbf{1.49}\mathrm{\textbf{e}}$\bf{-}$\textbf{07}$&$1.90\mathrm{e}$-$01$&$1.80\mathrm{e}$-$01$&$1.40\mathrm{e}$-$01$\\
				\hline
				Griewank&rate&$\textbf{89\%}$&$\textbf{90\%}$&$\textbf{96\%}$&$5\%$&$17\%$&$31\%$\\
				$(d=3)$&sol-err&$\textbf{1.11}\mathrm{\textbf{e}}$\bf{-}$\textbf{07}$
				&$\textbf{1.50}\mathrm{\textbf{e}}$\bf{-}$\textbf{07}$&$\textbf{6.32}\mathrm{\textbf{e}}$\bf{-}$\textbf{07}$&$1.72\mathrm{e}$+$01$&$1.11\mathrm{e}$+$01$&$1.35\mathrm{e}$+$01$\\
				~&fun-err&$\textbf{2.41}\mathrm{\textbf{e}}$\bf{-}$\textbf{08}$
				&$\textbf{4.11}\mathrm{\textbf{e}}$\bf{-}$\textbf{08}$&$\textbf{1.71}\mathrm{\textbf{e}}$\bf{-}$\textbf{09}$&$6.05\mathrm{e}$-$01$&$3.03\mathrm{e}$-$01$&$2.03\mathrm{e}$-$01$\\
				\hline
				Griewank&rate&$\textbf{85\%}$&$\textbf{91\%}$&$\textbf{98\%}$&$0\%$&$0\%$&$0\%$\\
				$(d=10)$&sol-err&$\textbf{1.94}\mathrm{\textbf{e}}$\bf{+}$\textbf{00}$
				&$\textbf{9.59}\mathrm{\textbf{e}}$\bf{-}$\textbf{06}$&$\textbf{7.04}\mathrm{\textbf{e}}$\bf{-}$\textbf{08}$&$2.24\mathrm{e}$+$01$&$5.18\mathrm{e}$+$01$&$2.20\mathrm{e}$+$01$\\
				~&fun-err&$\textbf{2.01}\mathrm{\textbf{e}}$\bf{-}$\textbf{01}$
				&$\textbf{4.92}\mathrm{\textbf{e}}$\bf{-}$\textbf{08}$&$\textbf{2.22}\mathrm{\textbf{e}}$\bf{-}$\textbf{11}$&$1.21\mathrm{e}$+$00$&$6.11\mathrm{e}$-$01$&$6.08\mathrm{e}$-$01$\\
				\hline
				Ackley&rate&$\textbf{100\%}$&$\textbf{100\%}$&$\textbf{100\%}$&$4\%$&$37\%$&$64\%$\\
				$(d=3)$&sol-err&$\textbf{7.50}\mathrm{\textbf{e}}$\bf{-}$\textbf{13}$
				&$\textbf{7.50}\mathrm{\textbf{e}}$\bf{-}$\textbf{13}$&$\textbf{7.25}\mathrm{\textbf{e}}$\bf{-}$\textbf{13}$&$8.72\mathrm{e}$-$02$&$1.23\mathrm{e}$-$03$&$3.77\mathrm{e}$-$04$\\
				~&fun-err&$\textbf{2.00}\mathrm{\textbf{e}}$\bf{-}$\textbf{06}$
				&$\textbf{2.00}\mathrm{\textbf{e}}$\bf{-}$\textbf{06}$&$\textbf{1.96}\mathrm{\textbf{e}}$\bf{-}$\textbf{06}$&$4.49\mathrm{e}$-$01$&$7.50\mathrm{e}$-$02$&$3.20\mathrm{e}$-$02$\\
				\hline
				Xin-She Yang $4$ &rate&$\textbf{94\%}$&$\textbf{100\%}$&$\textbf{100\%}$&$20\%$&$44\%$&$55\%$\\
				$(d=3)$&sol-err&$\textbf{9.87}\mathrm{\textbf{e}}$\bf{-}$\textbf{01}$
				&$\textbf{7.50}\mathrm{\textbf{e}}$\bf{-}$\textbf{13}$&$\textbf{7.50}\mathrm{\textbf{e}}$\bf{-}$\textbf{13}$&$3.95\mathrm{e}$+$00$&$1.17\mathrm{e}$-$03$&$3.67\mathrm{e}$-$03$\\
				~&fun-err&$\textbf{6.00}\mathrm{\textbf{e}}$\bf{-}$\textbf{02}$
				&$\textbf{1.50}\mathrm{\textbf{e}}$\bf{-}$\textbf{06}$&$\textbf{1.50}\mathrm{\textbf{e}}$\bf{-}$\textbf{06}$&$2.42\mathrm{e}$-$01$&$2.70\mathrm{e}$-$02$&$2.49\mathrm{e}$-$02$\\
				\hline
				Bartels Conn &rate&$\textbf{85\%}$&$\textbf{91\%}$&$\textbf{100\%}$&$7\%$&$26\%$&$33\%$\\
				~&sol-err&$\textbf{3.57}\mathrm{\textbf{e}}$\bf{-}$\textbf{06}$
				&$\textbf{7.05}\mathrm{\textbf{e}}$\bf{-}$\textbf{06}$&$\textbf{4.41}\mathrm{\textbf{e}}$\bf{-}$\textbf{07}$&$6.00\mathrm{e}$-$02$&$4.20\mathrm{e}$-$03$&$8.39\mathrm{e}$-$03$\\
				~&fun-err&$\textbf{6.78}\mathrm{\textbf{e}}$\bf{-}$\textbf{07}$
				&$\textbf{4.52}\mathrm{\textbf{e}}$\bf{-}$\textbf{07}$&$\textbf{7.21}\mathrm{\textbf{e}}$\bf{-}$\textbf{07}$&$4.68\mathrm{e}$-$02$&$2.12\mathrm{e}$-$03$&$4.21\mathrm{e}$-$03$\\
				\hline
				Schaffer $4$ &rate&$\textbf{93\%}$&$\textbf{100\%}$&$\textbf{100\%}$&$0\%$&$18\%$&$47\%$\\
				~&sol-err&$\textbf{3.29}\mathrm{\textbf{e}}$\bf{-}$\textbf{05}$
				&$\textbf{3.61}\mathrm{\textbf{e}}$\bf{-}$\textbf{06}$&$\textbf{3.10}\mathrm{\textbf{e}}$\bf{-}$\textbf{06}$&$1.94\mathrm{e}$-$02$&$2.90\mathrm{e}$-$03$&$1.05\mathrm{e}$-$04$\\
				~&fun-err&$\textbf{3.38}\mathrm{\textbf{e}}$\bf{-}$\textbf{07}$
				&$\textbf{3.65}\mathrm{\textbf{e}}$\bf{-}$\textbf{07}$&$\textbf{3.57}\mathrm{\textbf{e}}$\bf{-}$\textbf{07}$&$1.57\mathrm{e}$-$05$&$2.92\mathrm{e}$-$06$&$3.89\mathrm{e}$-$07$\\
				\hline
		\end{tabular}}
	\end{center}
\end{table*}
\begin{table*}[!thb]
	\caption{Comparisons between ESCBO and vanilla-CBO with initial particles of different distributions} \label{com2}
	\begin{center}
		\renewcommand\arraystretch{1.2}
		\resizebox{\textwidth}{!}{
			\begin{tabular}{llllllll} \hline
				&&\multicolumn {3}{c}{ESCBO}&\multicolumn {3}{c}{vanilla-CBO}\\
				\cmidrule(lr{0pt}){3-5}\cmidrule(lr{0pt}){6-8}
				&&$\mathcal{U}(-3,3)$&$\mathcal{U}(2,6)$&$\mathcal{N}_d(0,3)$&$\mathcal{U}(-3,3)$&$\mathcal{U}(2,6)$&$\mathcal{N}_d(0,3)$\\
				\hline
				Rastrigin&rate&$\textbf{100\%}$&$\textbf{100\%}$&$\textbf{100\%}$&$91\%$&$0\%$&$74\%$\\
				$(d=2)$&sol-err&$\textbf{6.03}\mathrm{\textbf{e}}$\bf{-}$\textbf{12}$
				&$\textbf{4.87}\mathrm{\textbf{e}}$\bf{-}$\textbf{12}$&$\textbf{4.79}\mathrm{\textbf{e}}$\bf{-}$\textbf{12}$&$6.15\mathrm{e}$-$05$&$4.15\mathrm{e}$+$00$&$2.30\mathrm{e}$-$05$\\
				~&fun-err&$\textbf{5.98}\mathrm{\textbf{e}}$\bf{-}$\textbf{10}$
				&$\textbf{4.83}\mathrm{\textbf{e}}$\bf{-}$\textbf{10}$&$\textbf{4.75}\mathrm{\textbf{e}}$\bf{-}$\textbf{10}$&$6.08\mathrm{e}$-$03$&$2.48\mathrm{e}$+$00$&$2.28\mathrm{e}$-$03$\\
				\hline
				Salomon&rate&$\textbf{100\%}$&$\textbf{100\%}$&$\textbf{100\%}$&$92\%$&$0\%$&$90\%$\\
				$(d=2)$&sol-err&$\textbf{4.65}\mathrm{\textbf{e}}$\bf{-}$\textbf{13}$
				&$\textbf{4.73}\mathrm{\textbf{e}}$\bf{-}$\textbf{13}$&$\textbf{4.90}\mathrm{\textbf{e}}$\bf{-}$\textbf{13}$&$6.22\mathrm{e}$-$06$&$3.69\mathrm{e}$+$00$&$5.03\mathrm{e}$-$06$\\
				~&fun-err&$\textbf{6.76}\mathrm{\textbf{e}}$\bf{-}$\textbf{08}$
				&$\textbf{6.86}\mathrm{\textbf{e}}$\bf{-}$\textbf{08}$&$\textbf{6.99}\mathrm{\textbf{e}}$\bf{-}$\textbf{08}$&$2.22\mathrm{e}$-$04$&$1.90\mathrm{e}$-$01$&$1.71\mathrm{e}$-$04$\\
				\hline
				Griewank&rate&$\textbf{100\%}$&$\textbf{82\%}$&$\textbf{100\%}$&$87\%$&$0\%$&$74\%$\\
				$(d=2)$&sol-err&$\textbf{4.93}\mathrm{\textbf{e}}$\bf{-}$\textbf{13}$
				&$\textbf{2.81}\mathrm{\textbf{e}}$\bf{-}$\textbf{07}$&$\textbf{4.65}\mathrm{\textbf{e}}$\bf{-}$\textbf{13}$&$1.32\mathrm{e}$-$05$&$3.00\mathrm{e}$+$01$&$8.33\mathrm{e}$-$05$\\
				~&fun-err&$\textbf{1.21}\mathrm{\textbf{e}}$\bf{-}$\textbf{13}$
				&$\textbf{1.41}\mathrm{\textbf{e}}$\bf{-}$\textbf{07}$&$\textbf{1.07}\mathrm{\textbf{e}}$\bf{-}$\textbf{13}$&$1.10\mathrm{e}$-$06$&$1.91\mathrm{e}$+$00$&$5.28\mathrm{e}$-$07$\\
				\hline
				Ackley&rate&$\textbf{100\%}$&$\textbf{100\%}$&$\textbf{100\%}$&$95\%$&$0\%$&$88\%$\\
				$(d=2)$&sol-err&$\textbf{5.00}\mathrm{\textbf{e}}$\bf{-}$\textbf{13}$
				&$\textbf{5.00}\mathrm{\textbf{e}}$\bf{-}$\textbf{13}$&$\textbf{5.00}\mathrm{\textbf{e}}$\bf{-}$\textbf{13}$&$2.55\mathrm{e}$-$05$&$2.53\mathrm{e}$+$00$&$4.48\mathrm{e}$-$04$\\
				~&fun-err&$\textbf{2.00}\mathrm{\textbf{e}}$\bf{-}$\textbf{06}$
				&$\textbf{2.00}\mathrm{\textbf{e}}$\bf{-}$\textbf{06}$&$\textbf{2.00}\mathrm{\textbf{e}}$\bf{-}$\textbf{06}$&$5.20\mathrm{e}$-$03$&$4.78\mathrm{e}$+$00$&$3.13\mathrm{e}$-$02$\\
				\hline
				Xin-She Yang $4$ &rate&$\textbf{100\%}$&$\textbf{62\%}$&$\textbf{100\%}$&$84\%$&$0\%$&$92\%$\\
				$(d=2)$&sol-err&$\textbf{5.00}\mathrm{\textbf{e}}$\bf{-}$\textbf{13}$
				&$\textbf{2.27}\mathrm{\textbf{e}}$\bf{+}$\textbf{01}$&$\textbf{5.00}\mathrm{\textbf{e}}$\bf{-}$\textbf{13}$&$1.23\mathrm{e}$-$06$&$2.76\mathrm{e}$+$01$&$4.02\mathrm{e}$-$05$\\
				~&fun-err&$\textbf{1.00}\mathrm{\textbf{e}}$\bf{-}$\textbf{06}$
				&$\textbf{5.00}\mathrm{\textbf{e}}$\bf{-}$\textbf{01}$&$\textbf{1.00}\mathrm{\textbf{e}}$\bf{-}$\textbf{06}$&$5.48\mathrm{e}$-$04$&$1.00\mathrm{e}$+$00$&$2.06\mathrm{e}$-$03$\\
				\hline
				Bartels Conn &rate&$\textbf{100\%}$&$\textbf{94\%}$&$\textbf{85\%}$&$37\%$&$0\%$&$25\%$\\
				~&sol-err&$\textbf{3.83}\mathrm{\textbf{e}}$\bf{-}$\textbf{07}$
				&$\textbf{5.40}\mathrm{\textbf{e}}$\bf{-}$\textbf{07}$&$\textbf{4.07}\mathrm{\textbf{e}}$\bf{-}$\textbf{07}$&$3.79\mathrm{e}$-$03$&$3.42\mathrm{e}$+$00$&$4.45\mathrm{e}$-$03$\\
				~&fun-err&$\textbf{6.91}\mathrm{\textbf{e}}$\bf{-}$\textbf{07}$
				&$\textbf{7.70}\mathrm{\textbf{e}}$\bf{-}$\textbf{07}$&$\textbf{6.12}\mathrm{\textbf{e}}$\bf{-}$\textbf{06}$&$1.90\mathrm{e}$-$03$&$5.16\mathrm{e}$+$00$&$2.23\mathrm{e}$-$03$\\
				\hline
				Schaffer $4$ &rate&$\textbf{100\%}$&$\textbf{100\%}$&$\textbf{100\%}$&$45\%$&$0\%$&$33\%$\\
				~&sol-err&$\textbf{2.64}\mathrm{\textbf{e}}$\bf{-}$\textbf{06}$
				&$\textbf{7.95}\mathrm{\textbf{e}}$\bf{-}$\textbf{06}$&$\textbf{1.30}\mathrm{\textbf{e}}$\bf{-}$\textbf{04}$&$3.53\mathrm{e}$-$04$&$1.64\mathrm{e}$+$00$&$6.84\mathrm{e}$-$03$\\
				~&fun-err&$\textbf{3.57}\mathrm{\textbf{e}}$\bf{-}$\textbf{07}$
				&$\textbf{3.61}\mathrm{\textbf{e}}$\bf{-}$\textbf{07}$&$\textbf{4.05}\mathrm{\textbf{e}}$\bf{-}$\textbf{07}$&$4.43\mathrm{e}$-$07$&$1.13\mathrm{e}$-$03$&$5.66\mathrm{e}$-$06$\\
				\hline
		\end{tabular}}
	\end{center}
\end{table*}

\subsection{Application on training Deep Neural Networks (DNNs)}\label{sec7.3}
In this section, we consider an optimization problem from machine learning \cite{Cui2020,Liu2023}. The architecture of DNNs and the selection of activation functions often lead to two prevalent issues in deep learning: gradient explosion and gradient vanishing. These issues impact both the effectiveness and the convergence rates of gradient-based methods. Nevertheless, the derivative-free algorithm is an effective choice for solving the problem where the gradient cannot be calculated explicitly. Training a DNN can be formulated as the following nonconvex finite-sum minimization problem:
\begin{equation}\label{dnn}
	\min_{W,b} \frac{1}{M}\sum_{m=1}^M\left\| \sigma(W_L\sigma(\cdots\sigma(W_1u_m+b_1)+\cdots)+b_L)-v_m\right\|^2, 
\end{equation}
where $\left\lbrace u_m\in\mathbb{R}^{N_0}\right\rbrace_{m=1}^M$ and $\left\lbrace v_m\in\mathbb{R}^{N_L}\right\rbrace_{m=1}^M$ are the given datasets, respectively, and $\sigma$ is a sigmoid activation function, i.e. $\sigma(s):=\frac{1}{1+\mathrm{e}^{-s}}$. The variables $\{W_\ell\in\mathbb{R}^{N_\ell\times N_{\ell-1}}\}_{\ell=1}^L$ and $\{b_\ell\in\mathbb{R}^{N_\ell}\}_{\ell=1}^L$ stand for the weight matrices and bias vectors, respectively, to be determined from the computation. The depth of this network is specified by the number of layers $L$, and the width is determined by the number of neurons $N_\ell$ at the $\ell$-th layer. 

Let $M=80$ and $M_{test}=20$ be the number of total training data and test data, respectively. The procedure for constructing a class of synthetic datasets is outlined below. We randomly generate the elements of weight matrices and bias vectors under distribution $\mathcal{N}(0,0.8)$ and the input data $\left\lbrace u_m\right\rbrace_{m=1}^{M+M_{test}}$ with $u_m\sim\mathcal{N}(a,\Sigma\Sigma^\top)$, where $a\in\mathbb{R}^{N_0}$ and $\Sigma\in\mathbb{R}^{N_0}$ are random vectors with standard Gaussian components. Then we obtain the output data $\left\lbrace v_m\right\rbrace_{m=1}^{M+M_{test}}$ by
\begin{equation*}
	v_m=\sigma(W_L\sigma(\cdots\sigma(W_1u_m+b_1)+\cdots)+b_L)+\tilde{v}_m,
\end{equation*}
where $\tilde{v}_m$ is the noise generated with the distribution $\mathcal{N}(0,0.0025)$. To assess the quality of the solution, it is common to evaluate it based on the following training error (\textbf{TrainErr}) and test error (\textbf{TestErr}):
\begin{equation*}
	\mbox{\textbf{TrainErr}}=\frac{1}{M}\sum_{m=1}^M \left\| \sigma(W_L\sigma(\cdots\sigma(W_1u_m+b_1)+\cdots)+b_L)-v_m\right\|^2, 
\end{equation*}
\begin{equation*}
	\mbox{\textbf{TestErr}}=\frac{1}{M_{test}}\sum_{m=M+1}^{M+M_{test}} \left\| \sigma(W_L\sigma(\cdots\sigma(W_1u_m+b_1)+\cdots)+b_L)-v_m\right\|^2.
\end{equation*}

Note that for high dimensional problems, such as those in deep learning, the computational cost of approximating gradient in the ESCBO algorithm is too high, if we perform \eqref{fd} for every particle. Hence, we introduce a new variant of the ESCBO algorithm with random mini-batch ideas to save the computational time and memory in practice. Specifically, we choose a sub-sample $S_k\subset[N]$ uniformly at random in each iteration, and then compute the gradient estimator by
\begin{equation}\label{fd1}
	g^{i,k}:=\begin{cases}
		\displaystyle 
		g(x^{i,k}),&\mbox{if} ~i\in S_k,\\
		\displaystyle 
		0,&\mbox{otherwise},
	\end{cases}
\end{equation}
where function $g$ is defined as in \eqref{fd}. We summarize the fast ESCBO (FESCBO) algorithm in Appendix \ref{secA1}.
We apply the FESCBO algorithm to solve \eqref{dnn} with the following parameters:
\begin{equation*}
	N=100, \quad\lambda=\delta=1,\quad \beta=1\mathrm{e}\mbox{\rm{+}}20,\quad \sigma=0.001, \quad b_k=10,\quad \alpha_k=0.99^k.
\end{equation*}
Initial points are sampled from a uniform distribution on $[-3,3]^d$ and the stopping criteria is set as in \eqref{stop1}. 

 \textbf{TrainErr} and \textbf{TestErr} averaged over $100$ simulations for training the DNNs with different architectures are listed in Table \ref{table4}. The dimensionality $d$ of variables in \eqref{dnn} can be calculated by $d=\sum_{l=1}^L N_{l-1}N_{l}+\sum_{l=1}^L N_l$. As the DNN expands in depth or width, $d$ significantly increases, leading to a multitude of potential local minimizers. As demonstrated in Table \ref{table4}, even if $d$ is large, the FESCBO algorithm still exhibits good performance on solving \eqref{dnn} in terms of both training and test errors.

\begin{table}[h!]
	\caption{Results of the FESCBO algorithm for training DNNs with different architectures}\label{table4}
	\centering
	\renewcommand\arraystretch{1.5}
	\begin{tabular}{ccccccccc} 
		\hline
		\multicolumn {6}{c}{DNN architecture}&\multirow{2}{*}{\makecell{Total dimensionality $d$	\\of the variables}} &\multirow{2}{*}{\textbf{TrainErr}}&\multirow{2}{*}{\textbf{TestErr}}\\
		\cmidrule(lr{0pt}){1-6}
		$L$&$N_0$&$N_1$&$N_2$&$N_3$&$N_4$\\

		\hline
		$2$&$5$&$10$&$1$&$0$&$0$&$71$&$3.57\mathrm{e}$-$05$&$5.28\mathrm{e}$-$05$\\
		$4$&$5$&$5$&$5$&$5$&$1$&$96$&$6.80\mathrm{e}$-$07$&$8.30\mathrm{e}$-$07$\\
		$4$&$5$&$10$&$10$&$10$&$1$&$291$&$3.25\mathrm{e}$-$07$&$6.10\mathrm{e}$-$07$\\
		\hline
		$2$&$10$&$10$&$1$&$0$&$0$&$121$&$2.18\mathrm{e}$-$04$&$5.76\mathrm{e}$-$04$\\
		$4$&$10$&$5$&$5$&$5$&$1$&$121$&$4.91\mathrm{e}$-$06$&$9.36\mathrm{e}$-$06$\\
		$4$&$10$&$10$&$10$&$10$&$1$&$341$&$1.22\mathrm{e}$-$05$&$1.73\mathrm{e}$-$05$\\
		\hline
		\end{tabular}
	\end{table}

\section{Conclusion}\label{sec8}
In this paper, we aim to find the global minimizer of a nonsmooth nonconvex optimization problem, where we focus on the cases that the gradient of the objective function can not be explicitly available. Inspired by this, we propose a novel algorithm with two schemes, where one is to use the CBO method to guarantee its efficiency for finding the global minimizer, and the other is to introduce an approximated gradient descent scheme to improve its robustness and accelerate its convergence rate. 
Unlike most existing work on the CBO algorithms, we rigorously analyze the convergence of the proposed algorithm to the global minimizer  without using mean-field limit model. To be more specific, we prove the convergence of the particles to a global consensus point and give an error estimate toward the global minimum as an optimality result of this consensus point. Moreover, we provide a sufficient condition on the objective function to guarantee the iteration complexity of the proposed algorithm, where we give that the expected mean square distance between the particles and the global minimizer can reach the accuracy threshold of $\varepsilon$ after at most $\mathcal{O}\left( \log(\varepsilon^{-1})\right)$ iterations. Finally, several experimental results demonstrate that, in comparison to the vanilla-CBO method, the proposed algorithm exhibits an enhanced ability to find a more accurate minimizer with a higher success rate, and has better practicability in training DNNs. 

\backmatter

\bmhead{Acknowledgements}

 The research of Wei Bian is partially supported by the National Natural Science Foundation of China grants (12425115, 12271127, 62176073). The research of Fan Wu is partially supported by the Postdoctoral Fellowship Program of CPSF under Grant Number GZC20233475.

\section*{Declarations}
\textbf{Data availability}  ~The data are available from the corresponding author on reasonable request.\\
\textbf{Conflict of interest} ~The authors declare that there is no Conflict of interest.

%


\begin{appendices}
	
	\section{Fast ESCBO algorithm}\label{secA1}
	\begin{algorithm}
		\caption{Fast Extra-Step Consensus-Based Optimization (FESCBO) algorithm}
		\begin{algorithmic}[1]
			\Require
			the number of particles $N$, initial points $x^{i,0}\in\mathbb{R}^d$, $i\in[N]$,
			parameters $\beta>0$, $\lambda> 0$, $\delta\geq0$, $\sigma>0$, mini-batch sizes $\{b_k\}$ and nonnegative step sizes $\left\lbrace \alpha_k\right\rbrace $.\\
			Set $k=0$.
			\While{a termination criterion is not met,}
			\State Compute the weighted average point by
			\begin{equation*}
				\bar{x}^{\star,k}=\left( \bar{x}^{\star,k}_1,\ldots,\bar{x}^{\star,k}_d\right)^\top =\frac{\sum_{i=1}^N x^{i,k}\mathrm{e}^{-\beta f(x^{i,k})}}{\sum_{i=1}^N \mathrm{e}^{-\beta f(x^{i,k})}}.
			\end{equation*}
			\State Generate a random vector $\eta^k=\left(\eta^k_1,\ldots,\eta^k_d \right)^\top $ whose components are i.i.d. and satisfies
			\begin{equation*}
				\eta_l^k\sim\mathcal{N}(0,\delta^2),\quad l\in[d].
			\end{equation*}
			\State Choose a random sample set $S_k\subset[N]$ with $|S_k|=b_k$ and compute the gradient estimator $g^{i,k}$ for $i\in[N]$ by \eqref{fd1}.
			\State Update $x^{i,k+1}$ for $i\in[N]$ by
			\begin{align*}
				&y^{i,k+1}=x^{i,k}-\lambda\left( x^{i,k}-\bar{x}^{\star,k}\right) -\left( x^{i,k}-\bar{x}^{\star,k}\right) \odot\eta^k, \\ 
				&x^{i,k+1}=y^{i,k+1}-\alpha_{k}g^{i,k}.
			\end{align*}
			\State Set $k=k+1$.
			\EndWhile
			\Ensure
			$x^{i,k+1}$, $i\in[N]$.
		\end{algorithmic}
	\end{algorithm}
	
		
\end{appendices}


\bibliography{sn-bibliography}

\end{document}